\documentclass{amsart}

\usepackage{graphicx}
\usepackage{a4wide}
\usepackage{psfrag}
\usepackage[all]{xy}

\let\cal\mathcal
\def\AA{{\cal A}}

\def\CC{{\cal C}}
\def\DD{{\cal D}}
\def\EE{{\cal E}}
\def\FF{{\cal F}}

\def\HH{{\cal H}}

\def\KK{{\cal K}}
\def\LL{{\cal L}}

\def\OO{{\cal O}}

\def\QQ{{\cal Q}}

\def\TT{{\cal T}}
\def\UU{{\cal U}}

\let\blb\mathbb

\def\bX{{\blb X}}

\def\bQ{{\blb Q}}

\def\bP{{\blb P}}
\def\bS{{\blb S}}

\def\bN{{\blb N}}
\def\bR{{\blb R}}

\def\bZ{{\blb Z}}

\let\frak\mathfrak

\def\aa{\frak{a}} 
 
\def\fK{\frak{K}} 

\def\modulo\mod

\def\soc{\operatorname{soc}}

\def\Mod{\operatorname{Mod}}

\def\mod{\operatorname{mod}}

\def\nilp{\operatorname{nilp}}

\def\qgr{\operatorname{qgr}}

\def\coh{\mathop{\text{\upshape{coh}}}}

\def\rad{\operatorname {rad}}

\def\rep{\operatorname{rep}}

\def\Ext{\operatorname {Ext}}

\def\Hom{\operatorname {Hom}}
\def\Homb{\operatorname {Hom}^{\bullet}}
\def\uhom{\operatorname {\mathit{hom}}}

\def\End{\operatorname {End}}
\def\RHom{\operatorname {RHom}}

\def\im{\operatorname {im}}

\def\cone{\operatorname {cone}}
\def\coker{\operatorname {coker}}
\def\ker{\operatorname {ker}}

\def\End{\operatorname {End}}

\def\rk{\operatorname {rk}}

\def\infrad{\operatorname {rad}^{\infty}}

\def\CYdim{\operatorname {CY dim}}

\def\ulin{\mathit{lin}}

\DeclareMathOperator{\Ind}{Ind}

\DeclareMathOperator{\ind}{ind}

\DeclareMathOperator{\cohproj}{cohproj}

\newcommand\C{C}

\newcommand\Kp{K^{+}}

\newcommand\D{D}
\newcommand\Db{{D^b}}

\newcommand\Hi{\operatorname{H^i}}

\renewcommand\t{\tau}

\newcommand\achi{\overline{\chi}}

\newtheorem{lemma}{Lemma}[section]
\newtheorem{proposition}[lemma]{Proposition}
\newtheorem{theorem}[lemma]{Theorem}
\newtheorem{corollary}[lemma]{Corollary}

\theoremstyle{definition}

\newtheorem{example}[lemma]{Example}
\newtheorem{definition}[lemma]{Definition}

{

}

\theoremstyle{remark}

\newtheorem{remark}[lemma]{Remark}

\newdimen\uboxsep \uboxsep=1ex
\def\uboxn#1{\vtop to 0pt{\hrule height 0pt depth 0pt\vskip\uboxsep
\hbox to 0pt{\hss #1\hss}\vss}}

\def\uboxs#1{\vbox to 0pt{\vss\hbox to 0pt{\hss #1\hss}
\vskip\uboxsep\hrule height 0pt depth 0pt}}

\def\Ob{\operatorname{Ob}}

\newcommand\exa{\nopagebreak \begin{center}\smallskip \nopagebreak               \begin{minipage}[t]{6cm}\sloppy}
\newcommand\exb{\end{minipage}\kern 1cm\begin{minipage}[t]{8cm}\sloppy}
\newcommand\exc{\end{minipage}\kern -3cm \smallskip\end{center}}

\title{Abelian hereditary fractionally Calabi-Yau categories}
\author{Adam-Christiaan van Roosmalen}
\address{Adam-Christiaan van Roosmalen\\Mathematisches Institut\\Universit\"{a}t Bonn
\\Endenicher Allee 60\\53111 Bonn\\Germany}\email{vroosmal@math.uni-bonn.de}
\begin{document}

\bibliographystyle{amsplain}

\begin{abstract}
As a generalization of a Calabi-Yau category, we will say a $k$-linear Hom-finite triangulated category is fractionally Calabi-Yau if it admits a Serre functor $\bS$ and there is an $n > 0$ with $\bS^n \cong [m]$.  An abelian category will be called fractionally Calabi-Yau is its bounded derived category is.  We provide a classification up to derived equivalence of abelian hereditary fractionally Calabi-Yau categories (for algebraically closed $k$).  They are:
\begin{itemize}
\item the category of finite dimensional representations of a Dynkin quiver, 
\item the category of finite dimensional nilpotent representations of a cycle,
\item the category of coherent sheaves on an elliptic curve or a weighted projective line of tubular type.
\end{itemize}
To obtain this classification, we introduce generalized 1-spherical objects and use them to obtain results about tubes in hereditary categories (which are not necessarily Calabi-Yau).
\end{abstract}

\maketitle

\tableofcontents

\section{Introduction}

As a generalization of Calabi-Yau categories, one can define fractionally Calabi-Yau categories as triangulated $k$-linear categories with a Serre functor $\bS$ where $\bS^n \cong [m]$ for some $n > 0$ (as triangle functors).  An abelian category $\AA$ is fractionally Calabi-Yau if its bounded derived category is fractionally Calabi-Yau.

The main result of this paper is the following classification result (Theorem \ref{theorem:FractionallyCY} in the text).  We will assume $\AA$ is $k$-linear over an algebraically closed field $k$.

\begin{theorem}\label{theorem:Introduction}
Let $\AA$ be an indecomposable abelian hereditary category which is fractionally Calabi-Yau, then $\AA$ is derived equivalent to either
\begin{enumerate}
  \item the category of finite dimensional representations $\rep Q$ over a Dynkin quiver $Q$, or
  \item the category of finite dimensional nilpotent representations $\nilp \tilde{A}_n$ where $\tilde{A}_n$ ($n \geq 0$) has cyclic orientation, or
  \item the category of coherent sheaves $\coh \bX$ where $\bX$ is either an elliptic curve or a weighted projective line of tubular type.
\end{enumerate}
\end{theorem}

The proof of this theorem resembles the proof of the classification of abelian 1-Calabi-Yau categories in \cite{vanRoosmalen08}, but the results we need to complete this classification are valid in greater generality.  We briefly outline some steps.

One of our main tools will be the twist functors $T_E$ and $T_E^*$ introduced in \cite{Seidel01}.  In \S\ref{section:TwistFunctors} we generalize the notion of an $n$-spherical object.  Roughly speaking, if $\AA$ is an Ext-finite abelian category with Serre duality, then a generalized 1-spherical object in $\Db \AA$ is an object $E$ such that $\End(E)$ is semi-simple and $\bS E \cong E[n]$.  We show that the associated twist functors are equivalences.  These twist functors correspond to tubular mutations and shrinking functors in the cases where the latter ones are defined.

Using these twist functors, we will prove some general facts about tubes (Auslander-Reiten components of the form $\bZ A_\infty / \langle \tau^n \rangle$, where $\t$ is the Auslander-Reiten translate) in a hereditary category $\AA$ with Serre duality.  Most notably, every indecomposable object $X$ in $\Db \AA$ with $\t^r X \cong X$ lies in a tube (Theorem \ref{theorem:TubeCriterium}) and tubes are convex (Theorem \ref{theorem:TubeMain}) in the sense that if there is a path $X_0 \to \cdots \to X_n$ in $\Db \AA$ with $X_0, X_n$ lying in the tube, then $X_i$ lies in the tube for all $i$.  Moreover, every tube contains a 1-spherical object.

Fix an abelian hereditary category $\AA$ which is fractionally Calabi-Yau of dimension 1 (as the other dimensions are easily dealt with).  In this case, there are no nonzero projective or injective objects and we have $\tau^n \cong 1$ for some $n >0$, where $\t$ is the Auslander-Reiten translate.  In this case, there will be generalized 1-spherical objects in $\Db \AA$.

We may assume that $\AA$ has at least two tubes and a nonzero map between them, as the other cases are again easy to deal with.  Furthermore, since the 1-Calabi-Yau case has already been classified in \cite{vanRoosmalen08}, we may assume $\tau \not\cong 1$.  This implies there is always an exceptional object (thus an object $L$ with $\Ext(L,L) = 0$).  We will start from such an exceptional object to construct a tilting object in $\Db \AA$ using twist functors as done in \cite{Lenzing07}.

Thus let $L$ be an exceptional object and let $E$ be a generalized 1-spherical object lying in a different tube, but with $\Hom(L,E) \not= 0$.  The Hom- and Ext-spaces between the objects $L$ and $E$ we chose should satisfy properties mimicking those of a line bundle and a generalized (semi-simple) 1-spherical object in a weighted projective line.

In $\Db \AA$ we then find a sequence $\EE = (T_E^i(L))_{i \in \bZ}$ and an associated $t$-structure with heart $\HH$ such that $\HH$ is derived equivalent with $\AA$, and the entire sequence $\EE$ lies in $\HH$.  There, the sequence $\EE$ is a coherent sequence in the sense of \cite{Polishchuk05}.  Furthermore, we will show that this sequence gives a Serre subcategory $\HH_0$ of $\HH$ such that $\HH / \HH_0 \cong \coh \bP^1$.

We can use this last result to obtain information about the subcategory of all objects of finite length in $\HH$ and use this to construct a tilting object in $\HH$.  We then use a well-known result by Happel where all hereditary categories with a tilting object are classified to conclude the proof of Theorem \ref{theorem:Introduction}.

As in \cite{vanRoosmalen08}, it is not hard to describe all hereditary categories which are derived equivalent to the categories in Theorem \ref{theorem:Introduction}.  We will do so in Theorem \ref{theorem:Equivalence} using split torsion theories.

\textbf{Acknowledgments} The author wishes to thank Michel Van den Bergh for meaningful discussions.
\section{Preliminaries}

\subsection{Notations and conventions}

Throughout, let $k$ be an algebraically closed field.  We will assume that all categories are $k$-linear and all algebras are $k$-algebras.

We will fix a Grothendieck universe $\UU$ and assume that all categories are $\UU$-categories i.e. all Hom-sets lie in $\UU$.  A category is called $\UU$-small if the set of objects is a set in $\UU$. 

A Hom-finite additive category $\AA$ is Krull-Schmidt.  A \emph{path} in $\AA$ is a sequence of indecomposable objects $X_1, X_2, \ldots, X_n$ with $\Hom(X_1,X_2) \not= 0$.  We will denote by $\ind \AA$ a set of chosen representatives of isomorphism classes of indecomposable objects.  If $\AA'$ is an additive subcategory of $\AA$, then we will choose $\ind \AA' \subseteq \ind \AA$.  If $\KK$ is an Auslander-Reiten component of $\AA$, then $\ind \KK$ refers to all indecomposables of $\ind \AA$ lying in $\KK$.

A nonzero abelian category $\AA$ will be called \emph{indecomposable} if it is not the coproduct of two nonzero categories.

If $Q$ is a quiver, then we denote by $\rep Q$ the category of finite dimensional representations.  The category $\nilp Q$ will be the full subcategory of (finite dimensional) nilpotent representations.

\subsection{Hereditary categories}
Let $\AA$ be an abelian category.  We will say that $\AA$ is \emph{Ext-finite} if $\dim \Ext^i(X,Y) < \infty$ for all $X,Y \in \Ob \AA$ and all $i \geq 0$.  This implies that $\dim \Hom(X,Y) < \infty$.  We will say that $\AA$ is \emph{hereditary} if and only if $\Ext^2(X,Y) = 0$ for all $X,Y \in \Ob \AA$.

For an abelian category $\AA$, we denote by $\Db \AA$ its bounded derived category.  The suspension or shift of $X$ will be denoted by $X[1]$.  There is a fully faithful functor $\AA \to \Db \AA$ mapping an object $X$ to the complex $X^\bullet$ such that $X^i$ is $X$ for $i=0$, and $0$ otherwise.  We will sometimes write $\AA[0]$ to refer to this full subcategory of $\Db \AA$.

If $\AA$ is a hereditary category, then we have the following description of the objects in the bounded derived category: $X \cong \oplus_{i \in \bZ} H^{i}(X)[-i]$, thus every object is isomorphic to the direct sum of its homologies.

\subsection{Serre duality}
Let $\CC$ be a Hom-finite triangulated $k$-linear category.  A \emph{Serre functor} \cite{BondalKapranov89} on $\CC$ is an auto-equivalence $\bS : \CC \to \CC$ such that for every $X,Y \in \Ob \CC$ there are isomorphisms
$$\Hom(X,Y) \cong \Hom(Y,\bS X)^*$$
natural in $X$ and $Y$, and where $(-)^*$ is the vector-space dual.

We will say $\CC$ has \emph{Serre duality} if $\CC$ admits a Serre functor.  An Ext-finite abelian category $\AA$ is said to satisfy Serre duality when the bounded derived category $\Db \AA$ admits a Serre functor.

It has been shown in \cite{ReVdB02} that $\CC$ has Serre duality if and only if $\CC$ has Auslander-Reiten triangles.  If we denote the Auslander-Reiten shift by $\t$, then $\bS \cong \t[1]$.

\subsection{Paths and split $t$-structures}\label{subsection:Splits}

Let $\AA$ be an abelian Ext-finite category (thus $\Db \AA$ is a Krull-Schmidt category).  A \emph{suspended path} in $\Db \AA$ is a sequence $X_0, X_1, \ldots, X_n$ of indecomposable objects such that for all $i=0, 1, \ldots, n-1$ we have $\Hom(X_i,X_{i+1}) \not= 0$ or $X_{i+1} \cong X_i[1]$.  If $\AA$ is indecomposable and not semi-simple, then it follows from \cite{Ringel05} that there is a suspended path from $X$ to $Y$ in $\Db \AA$ if and only if there is a path from $X$ to $Y$.  More generally \cite[Lemma 5]{Ringel05}, $\AA$ is indecomposable if and only if for every $X,Y \in \ind \Db \AA$ there is an $n \in \bZ$ and a suspended path from $X$ to $Y[n]$.  In this case, $\Db \AA$ is also called a \emph{block}.

We will say a full additive Krull-Schmidt subcategory $\UU$ of $\Db \AA$ is \emph{closed under successors} if and only if the following condition holds: if there is a suspended path from an indecomposable $X \in \Ob \UU$ to an indecomposable $Y \in \Ob\Db \AA$, then $Y \in \Ob \UU$.

We will use the following result from \cite{BergVanRoosmalen10}.

\begin{theorem}\label{theorem:Split}
Let $\AA$ be an indecomposable abelian Ext-finite category and let $\UU$ be full nonzero subcategory which is closed under successors.  If $\UU^\perp$ is nonzero, then there is a bounded $t$-structure on $\Db \AA$ given by $D^{\leq 0} = \UU$ and $\D^{\geq 0} = \UU^\perp[1]$ and the heart of this $t$-structure $\HH$ is hereditary and derived equivalent to $\AA$.
\end{theorem}

Following \cite{KellerVossieck88} we will say that a full additive subcategory $\UU$ of $\Db \AA$ is an \emph{aisle} if and only if
\begin{enumerate}
\item $\UU[1] \subset \UU$,
\item $\UU$ is closed under extensions; i.e. for each triangle $X \to Y \to Z \to X[1]$ we have $Y \in \UU$ if $X,Z \in \UU$
\item the inclusion $\UU \to \Db \AA$ admits a right adjoint.
\end{enumerate}

A full additive Krull-Schmidt subcategory $\UU$ of $\Db \AA$ closed under successors is an aisle.  The first two conditions are clear and the action of the right adjoint on objects is given as follows: $X \in \Ob \Db \AA$ get mapped to the largest direct summand of $X$ lying in $\UU$.

It was shown in \cite{KellerVossieck88} that there is a bijection between the aisles and the $t$-structures, given by $\DD^{\geq 0} = \UU$ and $\DD^{\leq 0} = \UU[1]^\perp$.

If $\Db \AA$ has Auslander-Reiten sequences, we will say that an aisle $\UU$ is \emph{$\t$-invariant} if $\t \UU \cong \UU$ as subcategories of $\Db \AA$.

\begin{proposition}\label{proposition:TauInvariant}
Let $\AA$ be an indecomposable Ext-finite abelian category with Serre duality.  Let $\UU$ be a $\t$-invariant aisle on $\Db \AA$ such that both $\UU$ and $\UU^\perp$ are nonzero.  The heart $\HH$ of the corresponding $t$-structure is hereditary and derived equivalent to $\AA$.  Furthermore, $\HH$ has no projective objects. 
\end{proposition}

\begin{proof}
For any $X \in \Db \AA$, there is a triangle $U \to X \to V \to U[1]$ where $U \in \UU$ and $V \in \UU^\perp$.  We have $\Hom(V,U[1]) \cong \Hom(\t^{-1} U, V)^* = 0$ since $\t^{-1} U \in \UU$.  This shows that $\UU$ induces a split $t$-structure on $\Db \AA$.  The rest follows easily from Theorem \ref{theorem:Split} and the fact that $\HH$ is closed under $\t$.
\end{proof}

\subsection{Ample sequences}
For the benefit of the reader, we will recall some definitions and results from \cite{Polishchuk05}.  Throughout, let $\HH$ be a Hom-finite abelian category.

We begin with the definition of an ample sequence.
\begin{enumerate}
\item A sequence $\EE = (L_i)_{i\in \bZ}$ is called \emph{projective} if for every epimorphism $X \to Y$ in $\HH$ there is an $n\in \bZ$ such that $\Hom(L_i,X) \to \Hom(L_i,Y)$ is surjective for $i < n$. 

\item A projective sequence $\EE = (L_i)_{i\in \bZ}$ is called \emph{coherent} if for every $X \in \Ob \HH$ and $n \in \bZ$, there are integers $i_1, \ldots, i_s \leq n$ such that the canonical map
$$\bigoplus_{j=1}^s \Hom(L_{i_j},X) \otimes \Hom(L_i, L_{i_j}) \longrightarrow \Hom(L_i,X)$$
is surjective for $i\ll 0$.

\item A coherent sequence $\EE =  (L_i)_{i\in \bZ}$ is \emph{ample} if for all $X \in \HH$, we have $\Hom(L_i,X) \not=0$ for $i\ll 0$.
\end{enumerate}

Let $A_{ij} = \Hom(L_i,L_j)$ for $i < j$, $A_{ii} = k$.  We may define an algebra $A = A(\EE) = \oplus_{i \leq j} A_{ij}$ in a natural way.  If $\EE$ is a coherent sequence, then $A$ is a \emph{coherent} $\bZ$-algebra in the sense of \cite{Polishchuk05} (see \cite[Proposition 2.3]{Polishchuk05}).

As usual, a graded right $A$-module $M$ is called \emph{coherent} if it is finitely generated and the kernel of any morphism $P \to M$ is finitely generated where $P$ is a finitely generated projective.  These modules form an abelian category $\coh A$ and the coherent modules which are bounded in grading form a Serre subcategory denoted by $\coh^{b} A$.  We define the quotient 
$$\cohproj A \cong \coh A / {\coh}^{b} A.$$

Finally, let $\EE = (L_i)_{i\in \bZ}$ be a sequence.  We will denote by $\HH_0$ the full subcategory of $\HH$ spanned by the objects $X \in \Ob \HH$ with the property that $\Hom(L_i,X) =0$ for $i\ll 0$.  If $\EE$ is projective, then $\HH_0$ is a Serre subcategory of $\HH$; if $\EE$ is ample then $\HH_0 = 0$.

We have the following general result.

\begin{theorem}\label{theorem:Polishchuk}
\cite[Theorem 2.4]{Polishchuk05} Let $\EE=(L_i)$ be a coherent sequence, $A=A(\EE)$ the corresponding $\bZ$-algebra, then there is a equivalence of categories $\HH / \HH_0 \cong \cohproj A$.
\end{theorem}

We will be interested in the special case where there is an automorphism $t : \Db \HH \longrightarrow \Db \HH$ such that $L_i \cong t^i L$.  We let $R = R(\EE) = \oplus_{i \in \bN} R_i$ where $R_i = \Hom(L,t^i L)$ and make it into a $\bZ$-graded algebra in an obvious way.

If $R$ is noetherian then the coherent $R$-modules correspond to the finitely generated ones and $\cohproj R$ corresponds to $\qgr R$, the category of finitely generated modules modulo the finite dimensional ones considered in \cite{ArtinZhang94}.

We will use following corollary of Theorem \ref{theorem:Polishchuk}.

\begin{corollary}\label{corollary:Polishchuk}
Let $\AA$ be a Hom-finite abelian category, $t$ be an autoequivalence of $\AA$, and $L$ an object of $\AA$.  If $\EE=(t^i L)$ is a coherent sequence and the corresponding graded algebra $R = R(\EE)$ is noetherian, then $\HH / \HH_0 \cong \qgr R$.
\end{corollary}

\subsection{$n$-Kronecker quivers}

Let $A$ be a finite dimensional algebra.  We will denote by $\mod A$ the category of finite dimensional right $A$-modules.  For a finite acyclic quiver $Q$, we will denote its path algebra by $kQ$.

An $n$-Kronecker quiver $K_n$ ($n \geq 1$) is a quiver with two vertices $\{x,y\}$ and $n$ arrows $x \to y$.  The path algebra $kK_n$ is hereditary, and is of finite representation type if $n=1$, is of tame type if $n=2$, and is of wild type is $n \geq 3$.  The $2$-Kronecker quiver is also called the Kronecker quiver and its path algebra is called the Kronecker algebra.

The following result is standard (\cite{Beilinson78}).  The functor $F$ below is given by $- \stackrel{L}{\otimes}_A (\OO_{\bP^1} \oplus \OO_{\bP^1}(1))$ where $A \cong \End (\OO_{\bP^1} \oplus \OO_{\bP^1}(1))$ is the Kronecker algebra.

\begin{theorem}
There is an exact equivalence $F:\Db \mod A \to \Db \coh \bP^1$ where $A$ is the Kronecker algebra and $F(A) \cong \OO_{\bP^1} \oplus \OO_{\bP^1}(1)$.
\end{theorem}

The category $\mod kK_2$ is well understood (see for example \cite{ARS}).  We will use the following.  Let $P_x,P_y \in \mod kK_2$ be indecomposable projective objects associated with the vertices $x,y$ of $K_2$.  Then $\dim \Hom(P_x,P_y) = 2$, every nonzero $f:P_y \to P_x$ is a monomorphism and the cokernel $R_f$ is a regular module with $\dim \Hom(P_x,R_f) = \dim \Hom(P_y,R_f) = 1$.  Moreover, $\Hom(R_f,R_g) = \Ext(R_f,R_g) = 0$ when $f,g \in \Hom(P_x,P_y)$ are linearly independent.

The algebra $kK_2$ is of tame type.  For algebras of wild type, we have the following proposition (see for example \cite{ARS}).

\begin{proposition}
A finite dimensional hereditary algebra $A$ is of wild type if and only if there is an indecomposable module $Z \in \mod A$ such that
$$\chi(Z,Z) = \dim \Hom(Z,Z) - \dim \Ext(Z,Z) < 0.$$
\end{proposition}
\section{Twist functors}\label{section:TwistFunctors}

Twist functors have appeared in the literature under different names, for example \emph{tubular mutations} \cite{Meltzer97}, \emph{shrinking functors} \cite{Ringel84}, and \emph{twist functors} \cite{Seidel01}.  Similar ideas were the mutations used in \cite{GorodentsevRudakov87} in the context of exceptional bundles on projective spaces and, more generally, in \cite{Bondal89}.

In this section, we will follow the general approach twist functors from \cite{Seidel01}, but introduce a small generalization.  If an object $E$ is endo-simple (i.e. $\dim \Hom(E,E) = 1$), then our definition of an associated twist functor coincides with that of \cite{Seidel01}.

Throughout, let $\AA$ be any abelian category.

\subsection{Definitions}

We start by introducing some notations.  Let $C(\AA)$ be the category of cochain complexes.  For $C,D \in \C (\AA)$, let $\uhom(C,D)$ be the complex in $\C(\Mod k)$ given by
$$\mbox{$\uhom^i(C,D) = \prod_{j \in \bZ} \Hom_{\AA}(C^j,D^{i+j})$ with $d^i_{\uhom(C,D)}(\phi) = d_D \circ \phi - (-1)^i \phi \circ d_C.$}$$

Note that the bi-functors $\Hom$ and $\uhom$ are not naturally isomorphic.  We do however have $\Hom^i(C,D) = \Hi \uhom(C,D)$, where $\Hom^i(C,D) = \Hom(C,D[i])$.

Let $A = \End(E)$ for an $E \in \C (\AA)$, and let $V$ and $W$ be complexes of right and left $A$-modules, respectively.  Denote by $D_A$ the functor $\Hom_{A-\Mod}(-,A)$.

We define the tensor product $V \otimes_A E$ and the complex of $A$-linear maps $\ulin_A (W,E)$ in the usual way, i.e. such that $(V \otimes_A E)^i = \coprod_{k+l=i} V^k \otimes_A E^l$ and $(\ulin_A(W,E))^i = \prod_{k+l=i} (D_A W^k) \otimes_A E^l$ (we will assume these objects exist in $\AA$).  The differentials are given by combining the differentials of $V$ and $E$ in such a way as to agree on an evaluation $W \otimes_k \ulin_A(W,E) \to A$.  We refer to \cite{Seidel01} for details.

If $C \to D$ is a map of complexes, then we denote by $\{C \to D\}$ the associated total complex, obtained by collapsing the bigrading.

Let $\AA$ be a $\UU$-small abelian category where $\UU$ is a Grothendieck universe (given the category $\AA$, we may choose $\UU$ such that $\AA$ is $\UU$-small).  There is a fully faithful embedding $\AA \to \Ind \AA$ where $\Ind \AA$ is the Grothendieck category of left exact functors $\AA \to \Mod k$ (see \cite{Gabriel62}) which lifts to a fully faithful embedding $\Db \AA \to \Db \Ind \AA$ (see \cite{LoVdB06}).  The category $\Ind \AA$ is a $\UU$-category and is closed under ($\UU$-indexed) products and coproducts.  The category $\Ind \AA$ is not to be confused with the set $\ind \AA$ of chosen representatives of isomorphism classes of indecomposable objects of $\AA$.

\begin{definition}\label{definition:K}
The category $\fK \subseteq \Kp (\Ind \AA)$ is defined to be the full subcategory whose objects are the bounded below cochain complexes $C$ of $\Ind \AA$-injectives which satisfy $H^i(C) \in \AA$ for all $i$, and $H^i(C)=0$ for $i \gg 0$.
\end{definition}

The following result is an analogue of \cite[Proposition 2.4]{Seidel01}.  We will follow its proof.

\begin{proposition}
There is an exact equivalence (canonical up to natural isomorphism) $\Db \AA \cong \fK$.
\end{proposition}

\begin{proof}
Since $\Ind \AA$ has enough injectives, it is clear that $\fK \cong \DD$ where $\DD \subset D^+ (\ind \AA)$ is the full subcategory of objects whose homologies satisfy the conditions of Definition \ref{definition:K}.  Again, $\DD$ is equivalent to $D^b_\AA(\Ind \AA)$, the full subcategory of $\D^+ (\ind \AA)$ consisting of these objects whose complexes are bounded and whose homologies lie in $\AA$.  The last equivalence $\Db \AA \to D^b_\AA(\Ind \AA)$ follows from \cite[Theorem 15.3.1]{KashiwaraShapira06}
\end{proof}

\subsection{Twist functors on $\fK$}

\begin{definition}
Let $E \in \Ob \fK$ be an object satisfying the following conditions
\begin{enumerate}
  \item[(S1)] $E$ is a bounded complex,
  \item[(S2)] for any $F \in \fK$, both $\Homb_\fK(E,F)$ and $\Homb_\fK(F,E)$ have finite (total) dimension over $k$.
\end{enumerate}
We will denote $\Hom(E,E)$ by $A$.  With $E$ we associate the \emph{twist functor}\index{twist functor} $T_E : \fK \to \fK$ defined by 
$$T_E F=\{\uhom(E,F) \otimes_A E \stackrel{\epsilon}{\rightarrow} F \}$$
where $\epsilon$ is the canonical map, and the \emph{dual twist functor} $T^*_E : \fK \to \fK$ defined by
$$T^*_E F=\{F \stackrel{\epsilon^*}{\rightarrow} \ulin_A(\uhom(F,E),E) \}$$
where $\epsilon^*$ is the canonical map.
\end{definition}

\begin{remark}
If direct sums of injectives in $\Ind \AA$ are again injectives, then $T_E F$ and $T^*_E F$ lie again in $\fK$.  In general, $T_E F$ and $T^*_E F$ might not lie in $K^+(\Ind \AA)$ but will satisfy the other conditions of Definition \ref{definition:K}.  We will fix an equivalence $D^b(\AA) \to \fK$ to resolve this.
\end{remark}

\begin{definition}
An object $E \in \fK$ satisfying (S1) and (S2) is called \emph{generalized $n$-spherical} for $n>0$ if
\begin{enumerate}
  \item[(S3)] $A \cong k^r$ as algebras and 
  $$\Hom^i_\fK(E,E) \cong \left\{ \begin{array}{ll}
    A & \mbox{if $i=0,n$} \\
    0 & \mbox{otherwise}
  \end{array} \right.$$
  as left and right $A$-modules.
  \item[(S4)] The composition $\Hom^j_\fK(F,E) \times \Hom^{n-j}_\fK(E,F) \to \Hom^n_\fK(E,E)$ is non-degenerate for all $j \in \bZ$ and $F \in \fK$, i.e. if $(f,-)$ or $(-,g)$ are zero maps then $f=0$ or $g=0$, respectively.
\end{enumerate}
\end{definition}

In case $r=1$, we recover the definition of $n$-spherical object as in \cite{Seidel01}.

\begin{remark}
As posed here, there is some redundancy in condition (S3).  It is sufficient that the isomorphisms for $\Hom^n_\fK(E,E)$ hold not as left and right $A$-modules, but as vector spaces.  That this induces isomorphisms as left and right $A$-modules then follows from (S4).
\end{remark}

\begin{lemma}\label{lemma:Lin}
Let $E$ be a generalized $n$-spherical object, and let $A = \End(E,E)$, then
$$H^\bullet \ulin_A(\uhom(F,E),\uhom(E,E) \otimes_A E) \cong D_A \Homb(F,E) \otimes_A \Homb(E,E) \otimes_A H^{\bullet} E$$
where $D_A$ is the dual with respect to $A$.
\end{lemma}

\begin{proof}
Since $A$ is semi-simple, we may write $\uhom(F,E) \cong H^\bullet \uhom (F,E) \oplus C$ as left $A$-modules, where $C$ is contractible.  We then find
\begin{eqnarray*}
H^\bullet \ulin_A(\uhom(F,E),\uhom(E,E) \otimes_A E) &\cong& H^\bullet \ulin_A(\Homb(F,E) \oplus C, \uhom(E,E) \otimes_A E) \\
&\cong& D_A \Homb(F,E) \otimes_A H^\bullet (\uhom(E,E) \otimes_A E) \\
&\cong& D_A \Homb(F,E) \otimes_A \Homb(E,E) \otimes_A H^\bullet E
\end{eqnarray*}
\end{proof}

We have the following generalization of \cite[Proposition 2.10]{Seidel01}

\begin{proposition}\label{proposition:Twist}
Let $E$ be a generalized $n$-spherical object for some $n > 0$, then the functors $T^*_E T_E$ and $T_E T^*_E$ are both naturally isomorphic to the identity functor on $\fK$. 
\end{proposition}

\begin{proof}
The proof is identical to the proof in \cite[Proposition 2.10]{Seidel01} using Lemma \ref{lemma:Lin} to calculate the homologies.
\end{proof}

\subsection{Twist functors on derived categories}

We will now translate the previous result to the setting we will be using in this article.  Let $\AA$ be an Ext-finite abelian category.

\begin{definition}
An object $E \in \Ob \Db \AA$ is called \emph{generalized $n$-spherical} for some $n > 0$ if it satisfies the following properties
\begin{enumerate}
  \item[(S1)] $E$ has a finite resolution by $\Ind \AA$-injectives,
  \item[(S2)] $\Homb(E,F)$ and $\Homb(F,E)$ have finite total dimension for all $F \in \Db \AA$,
  \item[(S3)] $A \cong k^r$ as algebras and 
  $$\Hom^i_{\Db \AA}(E,E) \cong \left\{ \begin{array}{ll}
    A & \mbox{if $i=0,n$} \\
    0 & \mbox{otherwise}
  \end{array} \right.$$
  as left and right $A$-modules,
  \item[(S4)] the composition $\Hom^j(F,E) \times \Hom^{n-j}(E,F) \to \Hom^n(E,E)$ is non-degenerate for all $j \in \bZ$ and $F \in \Db \AA$.
\end{enumerate}
where $A = \Hom(E,E)$.
\end{definition}

The third condition implies that $E \cong \oplus_{i = 1}^r E_i$ where $\Hom(E_i,E_j) \cong k$ if $i=j$ and $0$ otherwise, and that for every $i$ there is a unique $j$ such that $\Ext^n(E_i,E_j) \cong k$.  Thus there is a permutation $\sigma$ of $\{ 1, 2 \ldots, r\}$ defined by $\Ext^n(E_i,E_{\sigma(i)}) \cong k$.

\begin{lemma}\label{lemma:nondegenerate}
Let $\AA$ be an Ext-finite abelian category, and $E \in \Ob \Db \AA$ satisfying (S1), (S2), and (S3).  Then (S4) is equivalent to:
\begin{enumerate}
	\item[(S4')] there is a permutation $\sigma$ of $\{ 1, 2 \ldots, r\}$ such that $\Hom(E_i,F) \cong \Hom(F,E_{\sigma(i)}[n])^*$, natural in $F \in \Db \AA$.
\end{enumerate}
\end{lemma}

\begin{proof}
First assume that (S4) holds.  As above, this yields a non-degenerate pairing
$$\oplus_i \Hom(E_i,F) \times \oplus_i \Hom(F,E_{\sigma(i)}[n]) \to \oplus_i \Hom(E_i,E_{\sigma(i)}[n]).$$
From this it easily follows that the pairing 
$$\Hom(E_i,F) \times \Hom(F,E_{\sigma(i)}[n]) \to \Hom(E_i,E_{\sigma(i)}[n]) \cong k.$$
is nondegenerate as well.  This shows that $\Hom(F,E_{\sigma(i)}[n]) \cong \Hom(E_i,F)^*$.  The naturality is easily checked.

For the other direction, we see that the pairing
$$\oplus_i \Hom(E_i,F) \times \oplus_i \Hom(F,E_{\sigma(i)}[n]) \to \oplus_i \Hom(E_i,E_{\sigma(i)}[n])$$
obtained from (S4') is the one required in (S4).
\end{proof}

We now come to the main theorem of this section.

\begin{theorem}\label{theorem:TwistDerived}
Let $\AA$ be an Ext-finite abelian category, and let $E \in \Ob \AA$ be a generalized $n$-spherical object, then the twist functors $T_E$ and $T^*_E$ are quasi-inverses.
\end{theorem}

\begin{proof}
This is just a reformulation of Proposition \ref{proposition:Twist}.
\end{proof}

Note that $D_A \RHom(X,E) \cong \RHom(X,E)^*$ as right $A$-modules.  Writing $E \cong \oplus_{i = 1}^r E_i$ as above gives
\begin{eqnarray*}
\RHom(E,X) \otimes_A E &\cong& \oplus_{i=1}^r \RHom(E_i,X) \otimes_k E_i\\ 
\RHom(X,E)^* \otimes_A E &\cong& \oplus_{i=1}^r \RHom(X,E_i)^* \otimes_k E_i.
\end{eqnarray*}

In the rest of this article, we will be interested only in the case where $\AA$ is an Ext-finite abelian hereditary category with Serre duality.  In this case, the conditions (S1) and (S2) are automatic and the map $\sigma$ from Lemma \ref{lemma:nondegenerate} is induced by the Auslander-Reiten translate, i.e. $\t E_i \cong E_{\sigma (i)}$.  The conditions (S3) and (S4') then are equivalent to $\Hom(E,E) \cong k^r$ as algebras and $\bS E \cong E[1]$.  An example is given by $E \cong \oplus_{i=1}^r \t^i E_0$ where $\t^{r} E_0 \cong E_0$, and $\Hom(\t^i E_0, \t^j E_0) \cong k$ if and only if $i \equiv j \pmod r$ and $0$ otherwise (in particular $\Hom(E_0,E_0) \cong k$).  Every generalized 1-spherical object is a direct sum of objects of this form.

These conditions are satisfied, for example, when $E_0$ is a peripheral object of a generalized standard tube (see \S\ref{section:Tubes}), thus an Auslander-Reiten component of the form $\bZ A_\infty / \langle \tau^r \rangle$ for $r> 0$.  It will follow from Theorem \ref{theorem:TubeCriterium} that every direct summand of a generalized 1-spherical object in a hereditary category with Serre duality lies in a tube.

\begin{remark}
A generalized 1-spherical object in $\Db \AA$ is an example of a Calabi-Yau object in the sense of \cite{CibilsZhang09}.
\end{remark}

\begin{remark}
Let $\AA$ be a hereditary Ext-finite category.  Let $X,Y \in \Db \AA$ and write $A = \End(X)$.  There are triangles
$$T_X Y [-1] \to \RHom(X,Y) \otimes_A X \to Y \to T_X Y$$
and
$$T^*_X Y \to Y \to D_A \RHom(Y,X) \otimes_A X \to T^*_X Y [1].$$
\end{remark}

\begin{example}
Let $\AA = \nilp \tilde{A}_1$ where $\tilde{A}_1$ has cyclic orientation, and denote by $S_1$ and $S_2$ the two simple objects in $\ind \AA$.  An object $E \cong S_1 \oplus S_2$ is a generalized 1-spherical object.  Note that $A = \Hom(E,E) \cong k^2$ as algebras and $\Ext^1(E,E) \cong A$ as a left and right $A$-module, but not as an $A$-bimodule.
\end{example}
\section{Tubes}\label{section:Tubes}

Throughout, let $\AA$ be an indecomposable Ext-finite hereditary abelian category with Serre duality.  In this section, we will be interested in the stable components $\KK$ of the Auslander-Reiten quiver of $\AA$ or $\Db \AA$ of the form $\bZ A_\infty / \langle \tau^r \rangle$, called a \emph{tube} and we will refer to $r$ as the \emph{rank} of the tube.  If $r=1$, then $\KK$ is called a \emph{homogeneous tube}.

An indecomposable object $X \in \ind \AA$ is called \emph{peripheral} if the middle term $M$ in the almost split sequence $0 \to \t X \to M \to X \to 0$ is indecomposable.  The number of isomorphism classes of peripheral objects in a tube is given by the rank.

As usual, we will say a component is \emph{generalized standard} if $\infrad(X,Y)=0$, for all $X,Y \in \ind \KK$.  The generalized standard tubes occur, for example, in the category of finite dimensional representations of tame algebras and in the category of coherent sheaves on a smooth projective curve.

Recall that we will say that a $t$-structure on $\Db \AA$ is $\t$-invariant if the heart $\HH[0] \subset \Db \AA$ is invariant under $\t$.  If $\HH$ is hereditary, then the standard $t$-structure is $\t$-invariant if and only if $\HH$ does not have nonzero projective or injective objects.

The following two theorems are the main results of this section.

\begin{theorem}\label{theorem:TubeCriterium}
Let $\AA$ be a hereditary abelian category with Serre duality.  An Auslander-Reiten component in $\Db \AA$ is a tube if and only if it contains an indecomposable object $X$ such that $\t^r X \cong X$, for $r \geq 1$.
\end{theorem}

\begin{theorem}\label{theorem:TubeMain}\label{theorem:SimpleTube}\label{theorem:DirectingTubes}
Let $\AA$ be a hereditary abelian category with Serre duality and let $\KK$ be a tube in $\Db \AA$.  Then
\begin{enumerate}
  \item $\KK$ is generalized standard,
  \item $\KK$ is convex in the sense that if there is a path $X_0 \to \cdots \to X_n$ in $\Db \AA$ with $X_0, X_n \in \KK$, then $X_i \in \KK$ for all $i$.
  \item There exists a $\t$-invariant $t$-structure on $\Db \AA$ with hereditary heart $\HH \supseteq \KK$ such that the peripheral objects of $\KK$ are simple in $\HH$.
\end{enumerate}
\end{theorem}

A tube in a hereditary abelian category where the peripheral objects are all simple (as in Theorem \ref{theorem:TubeMain}(3)) will be called a \emph{simple tube}.  Note that a simple tube is generalized standard (This follows from Proposition \ref{proposition:SphericalInTube}; alternatively this can easily be shown using the additivity of $\infrad(X,-)$ and $\infrad(-,Y)$ on Auslander-Reiten sequences together with $\infrad(S_i,S_j) = 0$ for all simple objects).

For the proof of Theorem \ref{theorem:TubeMain} we will first show that both (2) and (3) hold under the assumption that $\KK$ is standard, and conclude by proving (1).

\begin{lemma}\label{lemma:TauInvariant}
Let $E$ be a generalized 1-spherical object in $\AA$.  There is a $\t$-invariant $t$-structure on $\Db \AA$ such that $E$ is a semi-simple object in the heart $\HH$.
\end{lemma}

\begin{proof}
Consider the $\tau$-invariant aisle $\UU$ given by
$$\ind \UU = \{ X \in \ind \Db \AA \mid \mbox{There is a path from $E$ to $X$} \}.$$
We denote $\HH = \UU \cap \UU[1]^\perp$ the heart of the corresponding $t$-structure.  Proposition \ref{proposition:TauInvariant} shows that $\HH$ is hereditary, derived equivalent with $\AA$, and has no nonzero projective objects.

There is a twist functor $T_E : \Db \AA \longrightarrow \Db \AA$ with 
$$T_E(X) \cong \cone(\RHom(E,X) \otimes_A E \longrightarrow X)$$
where $A = \End(E)$.  Since $E$ is 1-spherical, this functor is an autoequivalence of $\Db \AA$ (see theorem \ref{theorem:TwistDerived}).

To prove that $E$ is semi-simple in $\HH$, we shall show that the canonical map $\epsilon : \Hom(E,X) \otimes_A E \longrightarrow X$ in $\HH$ is a monomorphism for every indecomposable $X \in \Ob \HH$.  Consider the following exact sequence in $\HH$
$$0 \longrightarrow {H^{-1} T_E (X)} \longrightarrow {\Hom(E,X) \otimes_A E} \stackrel{\epsilon}{\longrightarrow} {X} \longrightarrow  H^{0}T_E (X) \longrightarrow {\Ext(E,X) \otimes_A E} \longrightarrow 0$$
where $T_E: \Db \AA \to \Db \AA$ is the twist functor associated with $E$.  Since $\HH$ is hereditary and $T_E (X)$ is indecomposable, either $H^{-1} T_E (X)$ or $H^{0} T_E (X)$ is zero.  Seeking a contradiction, we shall assume that $\epsilon$ is not a monomorphism and hence $H^{-1} T_E (X) \in \HH$ is nonzero.  In particular, there is a path from $E$ to $H^{-1} T_E (X)$.

Since $T_E$ is an autoequivalence of $\Db \AA$ and there is a path from $E$ to $T_E (X) [-1]$, there must also be a path from $T_E^{-1}(E) \cong E$ to $X[-1]$, a contradiction.
\end{proof}

\begin{proof}[Proof of Theorem \ref{theorem:TubeMain}(3) (assuming $\KK$ is generalized standard).]
Let $r$ denote the rank of the generalized standard tube $\KK$.  Let $E_0$ be a peripheral object and denote $E = \oplus_{i = 1}^r \t^i E_0$.  Since $\KK$ is generalized standard, $E$ is a generalized 1-spherical object.  The requested property now follows from Lemma \ref{lemma:TauInvariant}.
\end{proof}

\begin{proposition}\label{proposition:SphericalInTube}
Let $E$ be a generalized 1-spherical object, then $E$ is the direct sum of a set of representatives of peripheral objects of generalized standard tubes.  The additive closure of all objects in a generalized standard tube in $\AA$ is equivalent to the category of nilpotent representations of an $\tilde{A}_n$-quiver with cyclic orientation.
\end{proposition}

\begin{proof}
Without loss of generality, we may assume that $E$ has no nontrivial direct summands which are generalized 1-spherical.

As in Lemma \ref{lemma:TauInvariant}, we will consider a $t$-structure such that $E$ is semi-simple in the heart $\HH$.  The full, exact, and extension-closed abelian category $\AA_E$ generated by $E$ is a hereditary length category with Serre duality and it follows from \cite[Proposition 1.8.2]{ChenHenning09} (see also \cite[Proposition 8.3]{Gabriel73}) that $\AA_E$ is derived equivalent to the category of nilpotent representations of an $\tilde{A}_n$-quiver with cyclic orientation.

The $\tau$-action on $\AA_E$ corresponds to the $\tau$-action on $\AA$ and hence the Auslander-Reiten quiver of $\AA_E$ is a component of the Auslander-Reiten quiver of $\AA$.  We see that $E$ lies in a generalized standard tube of $\AA$ and that the additive closure of all objects in this tube is equivalent to the category of nilpotent representations of an $\tilde{A}_n$-quiver with cyclic orientation.
\end{proof}

Let $\KK$ be a generalized standard tube.  Every $A$ in $\ind \KK$ is a finite number of extensions of the peripheral objects and the number of peripheral objects occurring in such a composition series of $A$ will be denoted $l_\KK A$; this corresponds to the normal length of $A$ in $\nilp \tilde{A}_r$.

\begin{lemma}\label{lemma:SphericalLength}
Let $\KK$ be a simple tube and write $E$ for the direct sum of all the peripheral (simple) objects in $\KK$.  Let $X$ be an indecomposable object of $\AA$ that does not lie in $\KK$.  If $\Hom(E,X) \not= 0$ for a peripheral object $E$ of $\KK$, then for any $l \in \bN$, $X$ has a subobject $Y$ with $l_\KK Y = l$.
\end{lemma}

\begin{proof}
Applying the twist functor $T_E$ to $X$ gives an exact sequence
$$0 \longrightarrow {\Hom(E,X) \otimes_A E} \longrightarrow {X} \longrightarrow  T_E (X) \longrightarrow {\Ext(E,X) \otimes_A E} \longrightarrow 0$$
which we can shorten to $0 \rightarrow {\Hom(E,X) \otimes_A E} \rightarrow {X} \rightarrow  X_1 \to 0$ where $X_1 = \im(X \to T_E(X))$.
Note that $X_1$ is non-zero because $X$ does not lie in the standard tube $\KK$.  Since $X$ is indecomposable it follows readily from the last exact sequence that $\Ext(X_1,E) \not= 0$ and thus $\Hom(E,X_1) \not= 0$.

Iteration shows we may find a subobject $Y$ of $X$ with $l_\KK Y = l$ for any $l \in \bN$.
\end{proof}

\begin{lemma}\label{lemma:LengthGivesHoms}
Let $\KK$ be a standard tube of rank $r$, and let $A, B \in \ind \KK$.  If $l = \min \{l_\KK A, l_\KK B\}$, then there is a $k \in \bN$ such that $\dim \Hom(\t^k A,B) \geq \frac{l}{r}$.
\end{lemma}

\begin{proof}
Recall from Proposition \ref{proposition:SphericalInTube} that $\KK$ corresponds to the category of nilpotent representations of the quiver $\tilde{A}_n$ with cyclic orientation.  We will work in this last category.

If $l_\KK A \leq l_\KK B$, we will choose $k$ such that the simple socle of $\t^k A$ is isomorphic to the simple socle of $B$, thus $\t^k A$ is a subobject of $B$.  In this case, $\dim \Hom(\t^k A,B) = \dim \End(\t^k A) \geq \frac{l}{r}$.

If $l_\KK A > l_\KK B$, then we choose $k$ such that the simple top of $A$ is isomorphic to the simple top of $B$, thus $B$ is a quotient object of $\t^k A$.  We find that $\dim \Hom(\t^k A,B) = \dim \End(B) \geq \frac{l}{r}$.
\end{proof}

\begin{proposition}\label{proposition:SimplesInTubes}
Let $\KK$ be a standard tube.  If there are Auslander-Reiten components $\KK_1$ and $\KK_2$ different from $\KK$ with maps from $\KK_1$ to $\KK$ and from $\KK$ to $\KK_2$, then $\KK$ is not a simple tube.
\end{proposition}

\begin{proof}
Seeking a contradiction, assume that $\KK$ contains a simple object.

Let $X_1 \in \KK_1$ and $X_2 \in \KK_2$ be objects mapping nonzero to and from an object $Y \in \KK$, respectively.  We obtain from Lemma \ref{lemma:SphericalLength} and its dual that for every $l \in \bN$, there is a quotient object $Y_1$ of $X_1$ and a subobject $Y_2$ of $X_2$, both lying in $\KK$, with $l_\KK Y_1 = l_\KK Y_2 = l$.

If we write $r$ for the rank of $\KK$, then Lemma \ref{lemma:LengthGivesHoms} yields $\dim \Hom(\t^k Y_1, Y_2) \geq \frac{l}{r}$, for a certain $k \in \bZ$.  Since $l$ may be chosen arbitrarily large, and since $\dim \Hom(\t^k X_1, X_2) \geq \dim \Hom(\t^k Y_1, Y_2)$, we find the required contradiction.
\end{proof}

\begin{proof}[Proof of Theorem \ref{theorem:TubeMain}(2) (assuming $\KK$ is standard).]
We choose a tilt $\HH$ of $\AA$ as in Theorem \ref{theorem:TubeMain}(3), thus $\KK$ corresponds to a simple tube in $\HH$.  It is clear that, if $X_0 \to X_1 \to \cdots \to X_n$ is a path in the original abelian category, then it is a path in the tilted category $\HH$.

It follows from Proposition \ref{proposition:SimplesInTubes} that every object in this path lies in $\KK$.
\end{proof}

This final result states that an Auslander-Reiten component is a standard tube if and only it contains a finite $\t$-orbit, and finishes the proof of both Theorems \ref{theorem:TubeMain} and \ref{theorem:TubeCriterium}.

\begin{lemma}\label{lemma:TubeCriterium}
Let $\KK$ be a component of the Auslander-Reiten quiver such that there is an $X \in \ind \KK$ with $\t^r X \cong X$ for $r > 0$, then $\KK$ is a standard tube.
\end{lemma}

\begin{proof}
First, we use $X$ to define a $\t$-invariant aisle $\UU$ in $\Db \AA$ in the usual way:
$$\ind \UU = \{ Y \in \ind \Db \AA \mid \mbox{there is a path from $\t^n X$ to $Y$, for a certain $n \in \bN$} \}$$
and we denote the heart of the corresponding $t$-structure by $\HH$.  The translation $\t$ thus defines an exact autoequivalence $\HH \stackrel{\sim}{\rightarrow} \HH$.  We may assume $r$ is the smallest natural number such that $\t^r X \cong X$. 

We claim that there is a peripheral object $S$ lying in a standard tube such that there is a path from $X$ to $S$ and vice versa.  It will then follow from Theorem \ref{theorem:DirectingTubes}(2) that $X$ lies in the same Auslander-Reiten component as $S$ and hence that $\KK$ is a standard tube.

To show that $S$ is such an object, Proposition \ref{proposition:SphericalInTube} asserts we need only to verify that the $\t$-period of $S$ is finite and that $\sum_{i=0}^{s-1} \dim \Hom(S,\t^i S) = 1$ where $s > 0$ is the smallest natural number such that $S \cong \t^s S$ so that $\oplus_{i = 0}^{s-1} \t^i S$ is a 1-spherical object.  We will prove the existence of such an object $S$ by induction on $d = \sum_{i=0}^{r-1} \dim \Hom(X,\t^i X)$.

The case $d=1$ is trivial, thus assume $d > 1$.  In this case, there is a $j$ with $0 \leq j < r$ such that $\rad(X, \t^j X) \not= 0$.  Fixing isomorphisms $X \cong \t^r X$ and $\t^j X \cong \t^{r+j} X$, the functor $\t^r$ induces an automorphism of $\Hom(X, \t^j X)$.

Let $f \in \rad(X,\t^j X)$ be an eigenvector of this automorphism and denote $X_1 = \im f$.  Note that $\t^r X_1 \cong X_1$.  Furthermore, since $X_1$ is a quotient object of $X$ and a subobject of $\t^j X$, there are paths from $X$ to $X_1$ and vice versa.  Furthermore, as $\t$ is an autoequivalence of $\HH$, there are monomorphisms
$$\mbox{$\Hom(X_1,\t^{i} X_1) \hookrightarrow \Hom(X,\t^{i+j} X)$ for $0 \leq i < r$,}$$
such that, using that $f$ is a radical map, we find
$$\sum_{i=0}^{r-1} \dim \Hom(X_1,\t^i X_1) \leq \sum_{i=0}^{r-1} \dim \Hom(X,\t^{i+j} X) - 1,$$
as $1_X \in \Hom(X,X)$ is easily seen to not factor through $X_1$.

Changing $r$ in the left hand side to the smallest $r_1 > 0$ such that $X_1 \cong \t^{r_1} X_1$, we can use iteration to conclude the proof.
\end{proof}

The following corollary is our main application of Theorem \ref{theorem:DirectingTubes}.

\begin{corollary}\label{corollary:TubesDirecting}
Let $\AA$ be a hereditary category with Serre duality.  Let $X,Y \in \ind \AA$ and assume that $X$ or $Y$ lie in a tube of $\AA$, but $X$ and $Y$ do not lie in the same tube.  If $\Hom(X,Y) \not= 0$, then $\Ext(X,Y) = 0$.
\end{corollary}

\begin{proof}
Assume that $\Ext(X,Y) \not= 0$, then there is a short exact sequence $0 \to Y \to M \to X \to 0$.  It follows from \cite{Ringel05} that this gives a path from $Y$ to $X$, contradicting Theorem \ref{theorem:DirectingTubes}.
\end{proof}
\section{Hereditary fractionally Calabi-Yau categories}\label{section:Examples}

Let $\AA$ be an Ext-finite abelian category.  We will say that $\AA$ is Calabi-Yau of dimension $m$ (or $m$-Calabi-Yau) if $\Db \AA$ has a Serre functor $\bS : \Db \AA \to \Db \AA$, and $\bS \cong [m]$ as triangle functors.

By extension, if $\bS^n \cong [m]$ for some $n > 0$, then we will say $\AA$ is \emph{fractionally Calabi-Yau of dimension $\frac{m}{n}$} or $\frac{m}{n}$-Calabi-Yau.  We will write $\CYdim \AA = \frac{m}{n}$.

\begin{remark}
When we say $\AA$ is $\frac{m}{n}$-Calabi-Yau, we will mean that $\bS^n \cong [m]$.  In other words, we do not simplify the fraction.

It will however be convenient to interpret the fractional Calabi-Yau dimension $\CYdim \AA$ as a rational number and thus when we refer to the fractional Calabi-Yau dimension $\CYdim \AA$ we will simplify the fraction.

Thus a category which is fractionally Calabi-Yau of dimension 1 is not necessarily a 1-Calabi-Yau category.  An $n$-Calabi-Yau category is the same as a fractionally $\frac{n}{1}$-Calabi-Yau category.  Also, a fractionally $\frac{m}{n}$-Calabi-Yau category is a fractionally $\frac{km}{kn}$-Calabi-Yau category for any $k > 0$; the converse does not hold (see Example \ref{example:fractionalCYTubes}).
\end{remark}

\begin{example}\label{example:fractionalCalabi-Yau}\cite{Keller05}
Let $Q$ be a Dynkin quiver, and let $h$ be its Coxeter number.  In $\Db \rep Q$ we find
$$\bS^h = (\t [1])^h = \t^h [h] = [h-2]$$
and hence $\Db \rep Q$ is Calabi-Yau of fractional Calabi-Yau dimension $\frac{h-2}{h}$.
\end{example}

\begin{example}\label{example:fractionalCYTubes}
Let $Q$ be the quiver $\tilde{A_n}$ with cyclic orientation.  In the category $\Db \nilp Q$, we have $\bS^{n+1} = [n+1]$, hence $\nilp Q$ is $\frac{n+1}{n+1}$-Calabi-Yau.  It will only be 1-Calabi-Yau if $Q$ is the one-loop quiver.
\begin{figure}
	\centering
		\includegraphics[totalheight=0.40\textheight]{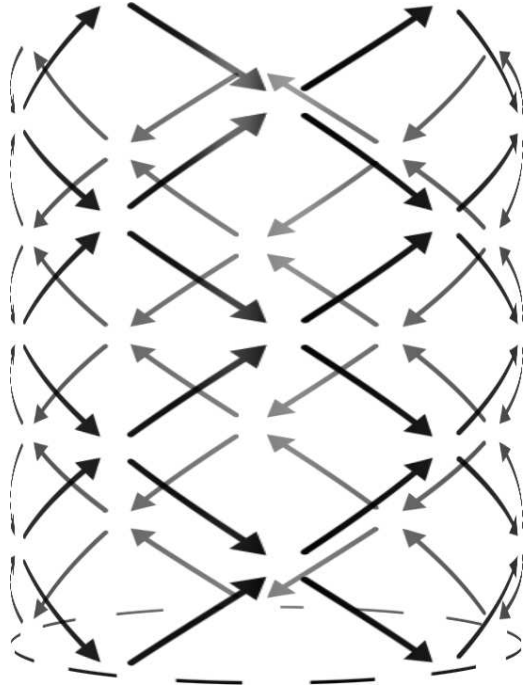}
	\caption{Auslander-Reiten quiver of $\nilp \tilde{A}_4$}
	\label{figure:LargerTube}
\end{figure}
\end{example}

\subsection{Calabi-Yau dimensions}

When $\AA$ is a Calabi-Yau category, it is well-known that the Calabi-Yau dimension coincides with the global dimension.  As Example \ref{example:fractionalCalabi-Yau} shows, we cannot expect an equally strong result in the case of fractionally Calabi-Yau categories.  The following example shows that the global dimension of a fractionally Calabi-Yau category might be arbitrarily high, even if the fractional Calabi-Yau dimension is not larger than one.

\begin{example}
Let $Q$ be the quiver $A_n$ $(n \geq 3)$ with zero relations on the composition of every two arrows, then $\rep Q$ has global dimension $n-1$, yet it is derived equivalent to $\rep A_n$ (without relations), and as such is Calabi-Yau of fractional dimension $\frac{n-1}{n+1}$.
\end{example}

The following proposition shows that in our case the fractional Calabi-Yau dimensions all lie between 0 and 1.  

\begin{proposition}\label{proposition:FractionalCYDim}
Let $\AA$ be a hereditary abelian category which is fractionally Calabi-Yau of dimension $d$, then $0 \leq d \leq 1$.
\end{proposition}

\begin{proof}
Let $\bS^n \cong [m]$ for some $n > 0$, thus $d = \frac{m}{n}$.  We find $\t^n \cong [m-n]$.  Since for every $X \in \ind \Db \AA$ there is a path from $\t^n X \cong X[m-n]$ to $X$ and thus (using that $\AA$ is hereditary) we have $m-n \leq 0$, and hence $d \leq 1$.

For the lower bound of $d$, let $X \in \ind \AA[0]$.  If $X$ is projective in $\AA[0]$ then $\bS X$ is an injective in $\AA[0]$ and hence $\t X \in \ind \AA[-1]$, otherwise $\t X \in \ind \AA[0]$.  It now follows from iteration that $\t^a X \in \AA[-b]$ where $a \geq b$, hence $n \geq n-m$ and thus $d \geq 0$.
\end{proof}

\subsection{Weighted projective lines}\label{subsection:WeightedProjectivesLines}

Weighted projective lines were first introduced in \cite{GeigleLenzing87} as lines in a weighted projective plane.    They can be seen as generalizations of projective lines where finitely many points $x$ have been given a weight $p(x) \in \bN$ strictly larger than 1.

Following more recent treatments, we will define a weighted projective line $\bX$ through the attached abelian category of coherent sheaves $\coh \bX$.

\begin{definition}\cite[Theorem 1]{Lenzing97}
A connected Ext-finite abelian hereditary noetherian category $\AA$ with a tilting complex and no nonzero projectives is said to be a \emph{category of coherent sheaves $\coh \bX$ over a weighted projective line $\bX$}.
\end{definition}

In our classification of fractionally Calabi-Yau categories, we will use following well-known characterization of hereditary categories with a tilting object.

\begin{theorem}\label{theorem:TiltingObject}\cite{Happel01}
Let $\HH$ be an indecomposable Ext-finite abelian hereditary category which admits a tilting object.  Then $\AA$ is derived equivalent to either
\begin{enumerate}
  \item $\mod A$, where $A$ is a finite dimensional hereditary algebra, or
  \item $\coh \bX$ where $\bX$ is a weighted projective line.
\end{enumerate}
\end{theorem}

In the case of a weighted projective line, this tilting object can be chosen to be a canonical algebra introduced by Ringel in \cite{Ringel84}.  Such an algebra can be described as the path algebra of a quiver as in Figure \ref{figure:TiltingComplexWeighted} where $t \geq 2$ and with relations $f_{i}^{p_{i}} = f_{2}^{p_{2}} - \lambda_{i} f_{1}^{p_{1}}$, for all $2 \leq i \leq t$, where $\lambda_i \not= \lambda_j$ for $i \not= j$.

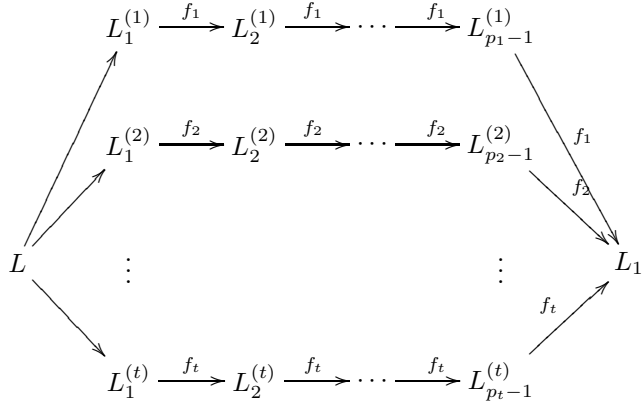
\begin{figure}
	\centering
		$$\xymatrix{
& L^{(1)}_{1} \ar[r]^{f_{1}} & L^{(1)}_{2} \ar[r]^{f_{1}} & \cdots \ar[r]^{f_{1}} & L^{(1)}_{p_{1} - 1} \ar[rdd]^{f_{1}} & \\
& L^{(2)}_{1} \ar[r]^{f_{2}} & L^{(2)}_{2} \ar[r]^{f_{2}} & \cdots \ar[r]^{f_{2}} & L^{(2)}_{p_{2} - 1} \ar[rd]^{f_{2}}& \\
L \ar[ruu] \ar[ru] \ar[dr]&  \vdots & && \vdots & L_1 \\
& L^{(t)}_{1} \ar[r]^{f_{t}} & L^{(t)}_{2} \ar[r]^{f_{t}} & \cdots \ar[r]^{f_{t}} & L^{(t)}_{p_{t} - 1} \ar[ru]^{f_{t}}&}$$
	\caption{The tilting complex of a weighted projective line.}
	\label{figure:TiltingComplexWeighted}
\end{figure}

The canonical algebra $\Lambda$ given by a $t$-tuple of weights $\mathbf{p}=(p_1, p_2, \ldots, p_t)$, and a $(t-2)$-tuple $\mathbf{\lambda} = (\lambda_3, \ldots, \lambda_t)$ of pairwise distinct elements of $k$, will be denoted by $\Lambda(\mathbf{p},\mathbf{\lambda})$.  We will call $\mathbf{p}$ the \emph{weight type} of $\Lambda$ and of the weighted projective line $\bX$.

Let $\bX$ be a weighted projective line of weight type $\mathbf{p}=(p_1, p_2, \ldots, p_t)$.  We will call
$$\chi_\HH = 2 - \sum_{i=1}^t \left( 1 -\frac{1}{p_i} \right)$$
the \emph{Euler characteristic} of $\HH = \coh \bX$.  We will only be interested in the case where $\chi_\HH = 0$.  Such a category $\HH$ will be called \emph{tubular}; every Auslander-Reiten component is a tube and the category $\HH$ is fractionally Calabi-Yau of dimension one (\cite{Kussin09}).

Note that $\chi_\HH = 0$ implies that the weight type of $\bX$ can only be one of the following: $(2,2,2,2), (3,3,3), (2,4,4),$ or $(2,3,6)$, thus the corresponding canonical algebra is as in Figure \ref{figure:Canonical}. By removing the final vertex from any of these quivers, one obtains an extended Dynkin quiver.

\begin{figure}[htb]
	\centering
		$$\xymatrix@R=10pt{
& \cdot \ar[rdd]\\
& \cdot \ar[rd]\\
\cdot \ar[ruu] \ar[ru] \ar[dr] \ar[ddr] & & \cdot\\
& \cdot \ar[ru]\\
& \cdot \ar[ruu]\\ 
& \cdot \ar[r] & \cdot \ar[rd]\\
\cdot \ar[ru] \ar[r] \ar[rd] & \cdot \ar[r] & \cdot \ar[r] & \cdot\\
& \cdot \ar[r] & \cdot \ar[ru]\\ 
& & \cdot \ar[rrd]\\
\cdot \ar[r] \ar[rru] \ar[dr]& \cdot \ar[r] & \cdot \ar[r] & \cdot \ar[r] & \cdot \\
& \cdot \ar[r] & \cdot \ar[r] & \cdot \ar[ru]\\ 
&&& \cdot \ar[rrrd]\\
\cdot \ar[rr] \ar[rrru] \ar[dr]&& \cdot \ar[rr] && \cdot\ar[rr] && \cdot \\
& \cdot \ar[r] & \cdot \ar[r] & \cdot \ar[r] & \cdot \ar[r] & \cdot \ar[ru]
}$$
	\caption{Canonical algebras with weight type $(2,2,2,2), (3,3,3), (2,4,4),$ and $(2,3,6)$, respectively.}
	\label{figure:Canonical}
\end{figure}
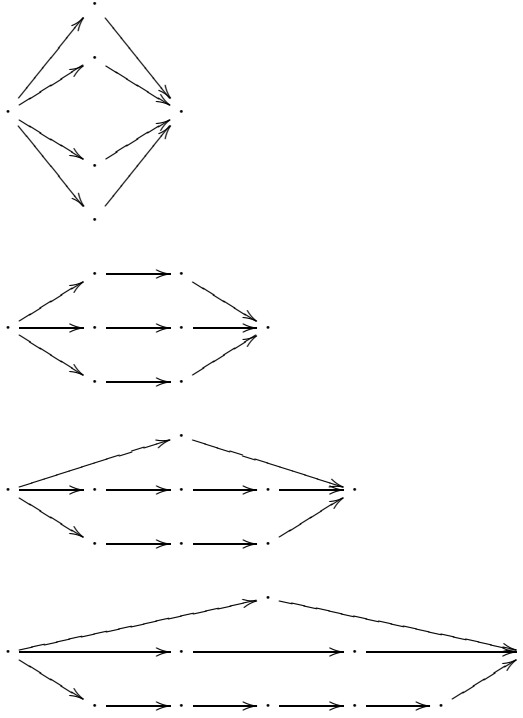

If $\chi_\HH < 0$, then $\HH$ is derived equivalent to the category of finitely presented modules over a tame hereditary algebra.  If $\chi_\HH > 0$, then $\HH$ is wild.  In neither of these two cases, $\HH$ is fractionally Calabi-Yau.

Let $\HH_f$ and $\HH_+$ denote the full subcategories of $\HH$ consisting of objects of finite length and those of all objects without simple subobjects, respectively.  Every object of $\HH$ is a direct sum of an object in $\HH_+$ and of an object in $\HH_f$. There are no nonzero morphisms from objects of $\HH_f$ to objects of $\HH_+$.  The category $\HH_f$ is an infinite coproduct of tubes.

Let $\chi(X,Y) = \dim \Hom(X,Y) - \dim \Ext(X,Y)$ be the usual Euler form.  Define the \emph{average Euler form}
$$\achi(X,Y) = \frac{1}{p} \sum_{j=0}^{p-1}\chi(\t^j X,Y)$$
where $p = \operatorname{lcm}\{p_1, p_2, \ldots, p_t\}$ having values in $\frac{1}{p} \bZ$.  For every $X \in \Ob \HH$, define the \emph{rank} and \emph{degree} as
\begin{eqnarray*}
\rk X &=& \achi(X,E) \\
\deg X &=& \achi(L,X) - \achi(L,L) \rk X
\end{eqnarray*}
where $E$ is a (generalized) 1-spherical object in $\HH_f$ and $L$ is an exceptional vector bundle with $\rk L = 1$.  Note that $\rk X \geq 0$.  The \emph{slope} of a nonzero $X$ is defined as $\mu(X) = \frac{\deg X}{\rk X} \in \bQ \cup \{\infty\}$.  The indecomposables $X$ with $\mu(X) =  \infty$ are exactly the ones of finite length.

A nonzero object $X$ is called \emph{semi-stable} if and only if $\mu X' \leq \mu X$ for every nonzero subobject $X'$ of $X$.  From this follows that $\mu X \leq \mu X''$ for every nonzero quotient object of $X''$.  In particular, if there is a nonzero map $X \to Y$ between semistable objects then then $\mu X \leq \mu Y$.

All semi-stable objects of the same slope $q$ together with the zero object form an exact abelian hereditary category.  We denote this category by $\HH^{(q)}$.  When $\chi_\HH = 0$ (thus $\bX$ is tubular), then every indecomposable object is semi-stable.  For $q,q' \in \QQ \cup \{\infty\}$, there is an equivalence $\HH^{(q)} \cong \HH^{(q')}$, and we have that $\HH^{(\infty)} \cong \HH_f$ as subcategories of $\HH$.

We have the following formula for weighted projective lines of tubular type
$$\achi(X,Y) = \rk X \deg Y - \deg X \rk Y.$$
Note that this implies that $\oplus_{j=0}^{p-1} \Hom(\t^j X,Y) \not= 0$ when $\mu X < \mu Y$.  Thus there is a morphism from every tube in $\HH^{(q)}$ to every tube in $\HH^{(q')}$ when $q < q'$.

For a more comprehensive treatment of weighted projective lines, we refer to \cite{Lenzing07} and the references therein.
\section{Classification of hereditary fractionally Calabi-Yau categories}

The proof of the classification of hereditary fractionally Calabi-Yau categories splits in two cases: one where the fractional Calabi-Yau dimension is strictly smaller than 1, and one where the fractional Calabi-Yau dimension is equal to 1.  The former case will be easily shown to be equivalent to $\mod Q$ where $Q$ is a Dynkin quiver.  We thus start by discussion the latter case.

\subsection{Categories with fractional Calabi-Yau dimension 1}

Let $\AA$ be an indecomposable hereditary category with fractionally Calabi-Yau dimension 1.  Thus there is an $n \in \bN$ such that $\t^n$ is isomorphic to the identity functor on $\AA$.  Theorems \ref{theorem:TubeCriterium} and \ref{theorem:TubeMain} imply that every Auslander-Reiten component is a generalized standard tube.  If there is only one such tube, then $\AA$ is equivalent to the category of finite dimensional nilpotent representations of an $\tilde{A}_n$-quiver with linear orientation (Proposition \ref{proposition:SphericalInTube}).  Thus we may assume there are at least two tubes in $\AA$.

We will furthermore assume that $\AA$ is not 1-Calabi-Yau.  The main reason for this last restriction is given in Proposition \ref{proposition:LS} below: such categories will always have an exceptional object.  We start with an observation.

\begin{corollary}\label{corollary:CalabiYauOnObjects}
Let $\AA$ be an abelian category which is not 1-Calabi-Yau, then there is an indecomposable object $X \in \ind \Db \AA$ such that $\t X \not\cong X$.
\end{corollary}

\begin{proof}
The proof of the classification of abelian 1-Calabi-Yau categories in \cite{vanRoosmalen08} does not use that $\bS \cong [1]$, but only that $\bS X \cong X[1]$ for all $X \in \Ob \Db \AA$.  In particular, if $\bS X \cong X[1]$ for all $X \in \Ob \Db \AA$ then $\AA$ is 1-Calabi-Yau.  If we assume that $\AA$ is not 1-Calabi-Yau, then there is an object $X \in \Db \AA$ such that $\t X \not\cong X$.
\end{proof}

\subsubsection{Choosing $L$ and $S$}  If the Auslander-Reiten component of $\AA$ has only one tube, then the classification follows easily from Proposition \ref{proposition:SphericalInTube}.  In case there is more than one tube, we will need to choose two elements as in the following proposition.

\begin{proposition}\label{proposition:LS}
Let $\AA$ be an indecomposable fractionally Calabi-Yau category with $\CYdim \AA = 1$, then every Auslander-Reiten component is a generalized standard tube.  If $\AA$ is not equivalent to $\mod \tilde{A}_n$, then there are two peripheral objects $L,S \in \ind \Db \AA$ lying in different Auslander-Reiten components with
\begin{enumerate}
  \item $\Ext(L,L) = 0$,
  \item $\Hom(L,S) \cong k$,
  \item $\Hom(L,\t^{-i} S)=0$ for $0 < i < s$,
\end{enumerate}
where $s$ is the rank of the Auslander-Reiten component containing $S$.
\end{proposition}

\begin{proof}
Since $\AA$ is fractionally Calabi-Yau of dimension $1$, there is an $n \in \bN$ such that $\bS^n \cong [n]$, and hence $\t^n \cong 1_{\AA}$.  It follows from Theorem \ref{theorem:TubeCriterium} that every Auslander-Reiten component is a generalized standard tube of rank at most $n$.

As stated in Corollary \ref{corollary:CalabiYauOnObjects}, we may assume that $\t$ is not the identity on objects, hence there is a peripheral object with $L \not\cong \t L$.  In particular, $\dim \Ext(L,L)=0$.

If $\AA$ is indecomposable and all objects lie in the same Auslander-Reiten component, then $\AA$ is equivalent to $\nilp \tilde{A}_n$ (see Proposition \ref{proposition:SphericalInTube}).  We will therefore assume that not all objects lie in the same Auslander-Reiten component, hence there is another peripheral object $X \in \ind \Db \AA$ such that $\Hom(L,X) \not= 0$ or $\Hom(X,L) \not= 0$, and $L$ and $X[n]$ lie in different tubes for all $n \in \bZ$.  Possibly by replacing $X$ by $\bS X$, we find that $\Hom(L,X) \not= 0$.

Let $E \cong \oplus_{i=1}^r \t^{i} X$ where $r$ is the rank of the tube containing $X$ and consider the object $Y = T^*_E L$, thus there is an exact triangle
$$Y \to L \to \RHom(L,E)^* \otimes E \to Y[1].$$
Since $E$ is generalized 1-spherical, we know that $T^*_E: \Db \AA \to \Db \AA$ is an autoequivalence and hence $Y$ is also a peripheral and exceptional object.  Furthermore, $L$ and $Y$ lie in different tubes so that $\Ext(Y,L) = 0$ by Corollary \ref{corollary:TubesDirecting}.

We may assume that $\Hom(Y,\t^i L)= 0$ when $0 < i < l$ where $l$ is the rank of the tube containing $L$.  Indeed, Let $0 < i < l$ be the smallest integer such that $\Hom(Y, {\t^i L}) \not= 0$.  It follows readily from applying the functor $\Hom(-,\t^j L)$ to the triangle
$$T^*_{\t^{i} L} Y \to Y \to \Hom(Y,\t^i L)^* \otimes_k \t^{i} L \to T^*_{\t^{i} L} Y[1]$$
that $\Hom(T^*_{\t^{i} L} Y,{\t^{j} L}) = 0$ when $0 < j \leq i$, and that $\Hom(T^*_{\t^{i} L} Y, L) \not= 0$.  Also if $Y$ does not have any indecomposable objects lying in the tube containing $L$, then neither does $T^*_{\t^{i} L} Y$.  Iteration shows that there is an object $Y'$ such that $Y'$ maps nonzero to $L$, but admits no nonzero maps to $\t$-shifts of $L$ which are not isomorphic to $L$.  Moreover, no indecomposable direct summand of $Y'$ lies in the tube containing $L$.

Furthermore, since $L$ is exceptional and thus lies in a nonhomogeneous tube, we infer that there is an indecomposable direct summand $M$ of $Y'$ which does not lie in a homogeneous tube and such that $\Hom(M,L) \not= 0$.  Note that $M$ and $L$ lie in different tubes so that Corollary \ref{corollary:TubesDirecting} shows that $\Ext(M,L) = 0$.  By possibly replacing $M$ by a peripheral object in the same tube, we may assume that the following conditions hold:
\begin{itemize}
\item $\Hom(M,L) \not= 0$,
\item $\Ext(M,L) = \Ext(L,M) = 0$ (since $\Ext(L,M) \cong \Hom(M,\t L)^* = 0$),
\item $M$ and $L$ are exceptional (since both are peripheral objects in nonhomogeneous tubes).
\end{itemize}
Hence $M \oplus L$ is a partial tilting complex in $\Db \AA$ and there is a fully faithful and exact functor $\Db \mod A \to \Db \AA$ where $A = \End(M \oplus L)$.  Note that $M \oplus L$ is a finite dimensional hereditary algebra.  Since every object $Z \in \ind \Db \AA$ lies in a tube, we know that $\dim \Hom(Z,Z) - \dim \Hom(Z,Z) \geq 0$.  The same thus holds for every indecomposable in $\Db \mod A$, and hence $A$ is not a wild hereditary algebra.  We infer that $\dim \Hom(M,L) \leq 2$.

If $\dim \Hom(M,L) = 1$, then we choose $L= L$ and $S = \bS M$.

If $\dim \Hom(M,L) = 2$, then $A$ is the path algebra of the Kronecker quiver.  The category $\mod A$ is well-known and there is an (indecomposable) regular module $S' \in \mod A$ such that $\dim \Ext(S',S') = \dim \Hom(S',S') = \dim \Hom(L,S') = 1$.  The image of $S'$ under the functor $\Db \mod A \to \Db \AA$ satisfies the same properties.  Thus there is a peripheral object $S$ in the tube containing this image such that $L$ and $S$ satisfy the required properties.
\end{proof}

We will assume $L$ and $S$ have been chosen as in Proposition \ref{proposition:LS} above.  We will denote by $s$ the rank of the tube containing $S$, and will write $E = \oplus_{i=1}^s \t^{-i} S$ and $A = \End E$.

\subsubsection{The autoequivalence $t: \Db \AA \to \Db \AA$} Let $L,S,$ and $E$ be as above.  As in the 1-Calabi-Yau case \cite{vanRoosmalen08}, we will consider the autoequivalence $t = T_E : \Db \AA \to \Db \AA$.  Note that since $L$ is a peripheral object not lying in the same tube as $S$, we know that $t L$ will lie in a different tube as $L$.

For every $i \in \bZ$, we have following triangle
\begin{equation}\label{equation:TriangleDefiningL}
t^i L \to t^{i+1} L \to \Ext(E,t^{i} L) \otimes_A E \to t^i L[1].
\end{equation}
In general, the $A$-modules $\Ext(E,t^i L)$ and $\Ext(E,t^j L)$ will not be isomorphic.  However, since $t S \cong \t^{-1} S$, we have $\Ext(E,t^i L) \cong \Ext(E,t^j L)$ as $A$-modules when $i-j$ is a multiple of $s$.  We also note that $t E \cong E$.

The triangles in (\ref{equation:TriangleDefiningL}) for different $i$'s are thus related by applying a certain power of the autoequivalence $t$.

\begin{remark}\label{Remark:LS}
Proposition \ref{proposition:LS} yields that $\Ext(E,t^{i} L) \otimes_A E \cong \Hom(t^{i} L, E)^* \otimes_A E \cong \t^{-i} S$.
\end{remark}

\subsubsection{The sequence $\EE$ and a $t$-structure in $\Db \AA$}\label{subsubsection:tFractional}

We will consider the sequence $\EE = (L_{i})_{i\in \bZ}$ where $L_i = t^{is} L$.  Theorem \ref{theorem:Split} shows the aisle $\UU$ given by
$$\ind \UU = \{ X \in \ind \Db \AA \mid \mbox{There is a path from $L_i$ to $X$, for an $i \in \bZ$} \}.$$
defines a $t$-structure with a hereditary heart $\HH$.  This does indeed define a bounded $t$-structure, as there are paths from $E[-1]$ to $L_{i}$ and thus by Theorem \ref{theorem:DirectingTubes}, we see that no direct summand of $E[-1]$ lies in $\ind D^{\leq 0}$.  It is now easy to see that $E \in \HH$ and $L_i \in \HH$, for all $i \in \bZ$.

The functor $t: \Db \AA \to \Db \AA$ restricts to an autoequivalence on $\HH$ which we will also denote by $t$.  By the choice of this $t$-structure, the triangles (\ref{equation:TriangleDefiningL}) correspond to short exact sequences (see Remark \ref{Remark:LS})
$$0 \to t^i L \to t^{i+1} L \to \t^{-i-1} S \to 0.$$
We find the following relations in $\HH$ between $L_i$ and $L_{i+1}$, for all $i \in \bZ$
$$0 \to L_i \to L_{i+1} \to B \to 0$$
where $B$ is given as an extension between the objects $(\t^{-i} S)_{0 \leq i < s}$.

\begin{lemma}\label{lemma:B}
The object $B$ as above is indecomposable and lies in $\TT_E$, the tube containing every indecomposable direct summand of $E$.    Furthermore, if $X \in \Ob \HH$ such that no direct summand of $X$ lies in $\TT_E$, then $\dim \Hom(X,B) = \dim \Hom(X,E)$.
\end{lemma}

\begin{proof}
As in Proposition \ref{proposition:SphericalInTube}, the additive closure of the objects in $\TT_E$ is equivalent to the category of finite dimensional nilpotent representations of a cyclic $\tilde{A}_n$-quiver.  Since $B$ has a composition series by elements of $(\t^{-i} S)_{0 \leq i < s}$, we know that $B$ lies in this subcategory.  Moreover, as noted above, every such object occurs exactly once in the composition series of $B$.  Applying the functor $\Hom(-,\t^{-i} S)$ to the short exact sequence
$$0 \to L_{-1} \to L_{0} \to B \to 0$$
shows that $\Hom(B,\t^{-i} S) = 0$ when $i \not= 0$.  This shows that $B$ is indecomposable.

Let $X$ be an indecomposable object in $\HH$ which does not lie in $\TT_E$.  Corollary \ref{corollary:TubesDirecting} shows that either $\Ext(X,Y) = 0$ or $\Hom(X,Y) = 0$, for all $Y$ in $\TT_E$.  We thus find that $\dim \Hom(X,B) = \sum_{i=1}^s \dim \Hom (X, \t^{-i} S) = \dim \Hom(X,E) = \dim \Hom(X,E)$.
\end{proof}

\begin{lemma}\label{lemma:EE2} Let $\EE = (L_i)_{i\in \bZ}$ and $E$ as above, then
\begin{enumerate}
\item $\Hom(L_i,L_j) = 0$ for $i>j$,
\item $\Ext(L_j,L_i) = 0$ for all $i+1 \geq j$,
\item $\Hom(E,L_i) = \Ext(L_i,E) = 0$ for all $i \in \bZ$.
\end{enumerate}
\end{lemma}

\begin{proof}
The only equality that does not follow directly from Theorem \ref{theorem:TubeMain} is $\Ext(L_j,L_i) = 0$ for $i+1 = j$.  To proof this last equality, apply $\Hom(L_{i},-)$ to short exact sequence
$$0 \to L_{i-1} \to L_{i} \to B \to 0.$$
We find, by using that $\dim \Hom(L_i,L_i) = 1 = \dim \Hom(L_i,B)$ (see Lemma \ref{lemma:B}), that indeed $\Ext(L_{i},L_{i-1}) = 0$ as required.
\end{proof}

We will also use following lemma.

\begin{lemma}\label{lemma:LiX}
If $\Hom(L_i,X) \not= 0$, then $\Hom(L_j,X) \not= 0$ for all $j \leq i$.
\end{lemma}

\begin{proof}
As in the proof of \cite[Lemma 4.2]{vanRoosmalen08}.
\end{proof}

\begin{proposition}\label{proposition:AlmostAmple}
The sequence $\EE$ is coherent in $\HH$, and for any $X \in \ind \HH$, there is an $n$ such that $\Hom(\t^n L_i,X) \not= 0$, for $i \ll 0$.
\end{proposition}

\begin{proof}
The proof that $\EE$ is projective and coherent is analogous to the proof of \cite[Proposition 4.3]{vanRoosmalen08}.

To prove the last claim, we note that for all $0 \leq n < r$ the sequences $\t^n \EE = (\t^n L_i)_{i \in \bZ}$ are projective and coherent.  For any $X \in \ind \HH$, let $i^{(n)}_{1}, \ldots, i^{(n)}_{m_n}$ be as in the definition of coherence.  We will show that, if $\Hom(\t^n L_i,X) \not= 0$ for $i \ll 0$ and some $n$, the canonical map
\begin{equation}
\theta : \bigoplus_{n = 1} ^ {r} \bigoplus_{j=1}^{m_n} \Hom(\t^n L_{i^{(n)}_{j}},X) \otimes \t^n L_{i^{(n)}_{j}} \longrightarrow X,
\end{equation}
is an epimorphism.

To ease notation, let $C = \coker \theta$ and write $M = \bigoplus_{n = 1} ^ {r} \bigoplus_{j=1}^{m_n} \Hom(\t^n L_{i^{(n)}_{j}},X) \otimes \t^n L_{i^{(n)}_{j}}$.

There is a nonsplit exact sequence $0 \to \im \theta \to X \to C \to 0$ and thus, since $\im \theta$ is a quotient object of $M$, we have $\Ext(C,M) \not= 0$.  We find $\Hom(M,\t C) \not= 0$.

In particular, there are $n,k \in \bZ$ such that $\Hom(\t^n L_{i^{(n)}_k},C) \not= 0$.  Using the projectivity of $\t^n \EE$ and Lemma \ref{lemma:LiX} we know there is an $l \ll 0$ such that $\Hom(\t^n L_l,C) \not= 0$ and every map in $\Hom(\t^n L_l,C)$ lifts to a map in $\Hom(\t^n L_l,X)$.  Again using coherence, such a map should factor through $M$.  We may conclude $C=0$, and hence $\theta$ is an epimorphism.

\end{proof}

In general, the sequence $\EE$ will not be ample.  We will write $\HH_0$ for the full subcategory of $\HH$ spanned by all objects $X$ such that $\Hom(L_i,X)=0$ for $i \ll 0$.  It follows from Lemma \ref{lemma:LiX} that $\Hom(L_i,X)=0$ for all $i \in \bZ$.  Recall that, since $\EE$ is projective, $\HH_0$ is a Serre subcategory.

\subsubsection{Description of $\HH / \HH_0$}

We will now prove that the category $\HH / \HH_0$ is equivalent to the category $\coh \bP^1$ of coherent sheaves on $\bP^1$.  To do this we will show that $R = R(\EE) \cong k[x,y]$, and then apply Corollary \ref{corollary:Polishchuk}. 

\begin{lemma}\label{lemma:HomDimensions}
Let $\EE= (L_i)_{i \in I}$ and $E = \oplus_{i=1}^s \t^i S$ be as before.  If $j \geq i$, then
$$\dim \Hom(L_{i},L_{j}) = 1 + (j-i).$$
\end{lemma}

\begin{proof}
It follows from Lemma \ref{lemma:B} that $\dim \Hom(L_i,B) = 1$.  The rest follows from Lemma \ref{lemma:EE2} and the short exact sequence
$$0 \longrightarrow L_i \longrightarrow L_{i+1} \longrightarrow B \longrightarrow 0.$$
\end{proof}

\begin{proposition}\label{proposition:AmpleRingByProjectiveLine}
There is a fully faithful and exact functor $F:\Db \coh \bP^1 \to \Db \AA$ such that $F(\OO_{\bP^1}(n)) \cong L_n$ for all $n \in \bZ$.
\end{proposition}

\begin{proof}
It follows from Lemma \ref{lemma:EE2} that $\Ext(L_0 \oplus L_1, L_0 \oplus L_1) = 0$.  There is thus a fully faithful functor $F_1:\Db \mod A \to \Db \HH$ such that $F_1(A) \cong L_0 \oplus L_1$ and where $A = \End (L_0 \oplus L_1)$.  By Lemma \ref{lemma:HomDimensions} we know that $\dim \Hom(L_0,L_1) = 2$ so that $A \cong kQ$ where $Q$ is the Kronecker quiver.

Likewise there is an equivalence $F_2:\Db \mod A \to \Db \coh \bP^1$ such that $F_2(A) \cong \OO_{\bP} \oplus \OO_{\bP}(1)$.  Combining gives a fully faithful and exact functor $F: \Db \coh \bP^1 \to \Db \HH$ such that $F(\OO_{\bP^1}) \cong L_0$ and $F(\OO_{\bP^1}(1)) \cong L_1$.  Note that exactness shows that $B \cong k(P)$ for some $P \in \bP^1$.  Again using exactness, it follows from 
$$\OO_{\bP^1}(n) \to \OO_{\bP^1}(n+1) \to k(P) \to \OO_{\bP^1}(n)[1]$$
that $F(\OO_{\bP^1}(n)) \cong L_n$ for all $n \in \bZ$.
\end{proof}

\begin{corollary}\label{corollary:QuotientIsP}
We have $\HH / \HH_0 \cong \coh \bP^1$.
\end{corollary}

\begin{proof}
It follows from Proposition \ref{proposition:AmpleRingByProjectiveLine} that $R(\EE) \cong k[x,y]$ as graded rings (with the usual grading), so that Theorem \ref{corollary:Polishchuk} yields $\HH / \HH_0 \cong \coh \bP^1$.
\end{proof}

\subsubsection{Objects of finite length}

In order to construct a tilting object in $\HH$, we will need more information about the simple tubes in $\HH$.  To do this, we will discuss the subcategory $\HH_f$ of $\HH$ consisting of all objects of finite length, thus $\HH_f$ is the full subcategory of $\HH$ consisting of all simple tubes of $\HH$ (see Lemma \ref{lemma:SphericalLength}).

Denote by $\pi: \HH \to \HH / \HH_0$ be the usual quotient functor.

\begin{lemma}
For all $X \in \Ob \HH$ and all $i \in \bZ$, we have 
$$\dim \Hom_\HH(L_i,X) \leq \dim \Hom_{\HH / \HH_0}(\pi(L_i),\pi(X)).$$
\end{lemma}

\begin{proof}
We know that $\Hom_{\HH / \HH_0}(\pi(L_i),\pi(X)) = \lim\limits_{\to} (\Hom_{\HH}(L'_i,X/X'))$ where the direct limit is taken over all subobjects $L'_i$ of $L_i$ and all subobjects $X'$ of $X$ such that $L_i / L'_i, X' \in \HH_0$.  By the definition of $\HH_0$, we have $L'_i = L_i$.  From $X' \in \HH_0$ follows that $\Hom_\HH(L_i,X)$ is always a subspace of $\Hom_\HH(L_i,X/X')$.  The required inequality follows easily.
\end{proof}

\begin{proposition}\label{proposition:MapsToSimple}
Every nonzero object $X \in \Ob \HH$ maps nonzero to a simple object of $\HH$.
\end{proposition}

\begin{proof}
If $X$ has finite length, then the statement is trivial.  So assume $X$ has infinite length, so that there is a sequence of epimorphisms
$$X=X_0 \to X_1 \to X_2 \to \cdots \to X_k \to \cdots$$
where the kernels $K_k = \ker(X_k \to X_{k+1})$ are nonzero.  By Proposition \ref{proposition:AlmostAmple}, there is an $n \in \bN$ such that the sequence $(\t^n L_i)_{i \in \bZ}$ maps nonzero to $K_k$ for an infinite number of $k$'s.  We may choose notations such that this sequence is $(L_i)_{i \in \bZ}$.

Thus in $\HH / \HH_0$, there is an infinite sequence of epimorphisms
$$\pi X=\pi X_0 \to \pi X_1 \to \pi X_2 \to \cdots \to \pi X_k \to \cdots$$
where infinitely many morphisms are not invertible.  Indeed, the functor $\pi$ is exact and infinitely many kernels are nonzero.  This shows $\pi(X)$ has infinite length.

Since every short exact sequence in $\HH / \HH_0$ is induced from a short exact sequence in $\HH$ (\cite[Corollaire III.1]{Gabriel62}) and $E$ semi-simple, we know that $\pi E$ is semi-simple in $\HH / \HH_0$ as well.  In $\HH / \HH_0 \cong \coh \bP^1$ (see Corollary \ref{corollary:QuotientIsP}) every indecomposable object of infinite length maps nonzero to every simple object.  We find that $\Hom_{\HH / \HH_0}(\pi X,\pi E) \not= 0$ and claim this implies $\Hom_\HH(X,E) \not=0$.

By the definition of a quotient category, we find a subobject $X'$ of $X$ such that $\Hom_{\HH}(X',E) \not= 0$, thus either $\Hom_\HH(X,E) \not= 0$, and we are done, or $\Hom_\HH(E,X/X') \cong \Ext_\HH(X/X',E)^* \not= 0$.  In the latter case, since $E$ is semi-simple and thus lies in a simple tube, it follows from Proposition \ref{proposition:SimplesInTubes} that $X/X'$ also has an indecomposable direct summand in the same simple tube.  Thus $\Hom_\HH(X/X',E) \not= 0$ and hence also $\Hom_\HH(X,E) \not=0$.
\end{proof}

\begin{proposition}\label{proposition:UniqueSimple}
In every simple tube of $\HH$ there is exactly one peripheral object $T$ such that $\Hom(L,T) \not= 0$.
\end{proposition}

\begin{proof}
Let $\KK$ be a simple tube in $\HH$.  It follows from Proposition \ref{proposition:AlmostAmple} that there is at least one peripheral, and hence simple, object $T$ such that $\pi T \not \cong 0$.  In this case, $\pi T$ is also simple, hence it is a peripheral object in a simple tube in $\HH / \HH_0 \cong \coh \bP^1$.

Next, we show $\KK$ has only one such peripheral object.  Let $T' \in \ind \HH$ be a peripheral object of $\KK$ with $T \not\cong T'$ and $\pi T' \not\cong 0$.  Since $\KK$ corresponds to an abelian subcategory of $\HH$ equivalent with $\mod \tilde{A}_n$ where $\tilde{A}_n$ has cyclic orientation, there is an object $X \in \ind \HH$ with simple top $T$, simple socle $T'$, and $\dim \End(X) = 1$.  Note that $\Hom(L,T') \not= 0$ implies that $\Hom(L,X) \not= 0$ so that $\pi X \not\cong 0$.

We now readily verify that $\Hom(\pi X, \pi T) \not= 0$, $\Hom(\pi T', \pi X) \not= 0$, and $\dim \End(\pi X) = 1$.  The last statement shows $\pi X$ is indecomposable.  We have shown there is a path from $\pi T'$ to $\pi T$, and thus $\pi T'$ and $\pi T$ lie in the same simple tube of $\HH / \HH_0$, hence $\pi T \cong \pi T'$.  This implies $\Hom(T,T') \not= 0$ and hence $T \cong T'$.
\end{proof}

\subsubsection{A tilting object}

In this final step in the classification of indecomposable abelian hereditary categories which are fractionally Calabi-Yau of dimension 1 but not 1-Calabi-Yau, we will use the previous results to construct a tilting complex in the derived category.

Let $\AA$ be such a category.  We may choose objects $L$ and $S$ as in Proposition \ref{proposition:LS} and use these to find a hereditary category $\HH$ derived equivalent to $\AA$ as in \S\ref{subsubsection:tFractional}.  This object $L$ is exceptional, it will be our starting point of the tilting object.

Let $\bX$ be a set parameterizing the simple tubes of $\HH$, thus with every $x \in \bX$, there corresponds a unique $\TT_x$.  Proposition \ref{proposition:UniqueSimple} yields that such a tube has a unique simple object $S_x$ such that $\Hom(L,S_x) \not= 0$.  Denote by $s_x$ the rank of the tube $\TT_x$ and write $E_x = \oplus_{i=1}^{s_x} \t^{-i} S_x$.

With every simple tube $\TT_x$, there corresponds a twist functor $T_{E_x} : \Db \HH \to \Db \HH$, which restricts to an autoequivalence $\HH \to \HH$, also denoted by $T_{E_x}$.  We will write $T_{E_x}^i L$ as $L^{(x)}_i$.

\begin{lemma}
The set $\LL = \{L^{(x)}_i \mid x \in \bX, 0 \leq i \leq s_x \}$ forms a partial tilting set, i.e. $\Hom(A,B[n]) = 0$ for all $A,B \in \LL$ and $n \not= 0$.  Moreover, for all $X \in \HH$, there is an $A \in \LL$ such that $\Hom(A,X) \not= 0$ or $\Ext(A,X) \not= 0$.
\end{lemma}

\begin{proof}
This follows easily from the exact sequences
$$0 \to L^{(x)}_i \to L^{(x)}_{i+1} \to \Ext(E_x,L^{(x)}_i) \otimes_{A_x} E_x \to 0,$$
where $A_x = \End(E_x)$, together with Proposition \ref{proposition:MapsToSimple}.
\end{proof}

Since there are nonzero maps from $L^{(x)}_{s_x}$ to $L^{(y)}_{s_y}$ for all $x,y \in \bX$, we see that $L^{(x)}_{s_x} \cong L^{(y)}_{s_y}$.  We may sketch the partial tilting set as in Figure \ref{figure:CanonicalObject} where $\dim \Hom(L,L_1)= 2$.

\begin{figure}
	\centering
		$$\xymatrix{
& L^{(x_1)}_{1} \ar[r]^{f_{x_1}} & L^{(x_1)}_{2} \ar[r]^{f_{x_1}} & \cdots \ar[r]^{f_{x_1}} & L^{(x_1)}_{s_{x_1} - 1} \ar[rdd]^{f_{x_1}} & \\
& L^{(x_2)}_{1} \ar[r]_{f_{x_2}} & L^{(x_2)}_{2} \ar[r]_{f_{x_2}} & \cdots \ar[r]_{f_{x_2}} & L^{(x_2)}_{s_{x_2} - 1} \ar[rd]_{f_{x_2}}& \\
L \ar[ruu]^{f_{x_1}} \ar[ru]_{f_{x_2}} \ar[dr]_{f_{x_t}}&  \vdots & && \vdots & L_1 \\
& L^{(x_t)}_{1} \ar[r]_{f_{x_t}} & L^{(x_t)}_{2} \ar[r]_{f_{x_t}} & \cdots \ar[r]_{f_{x_t}} & L^{(x_t)}_{s_{x_t} - 1} \ar[ru]_{f_{x_t}}& \\
& & \vdots & \vdots &}$$
\caption{The partial tilting set $\LL$}
\label{figure:CanonicalObject}
\end{figure}

In order to show this is a tilting object in $\Db \AA$, we need to show $\LL$ has only finitely many elements, or equivalently, that $\HH$ has only finitely many nonhomogeneous simple tubes (each nonhomogeneous tube gives an ``arm'' in $\LL$).  Let $\aa$ be the preadditive subcategory of $\Db \AA$ generated by $\LL$.  Following \cite[Theorem 5.1]{vanRoosmalen06}, there is a fully faithful functor $\Db \mod \aa \to \Db \AA$ where $\mod \aa$ is the category of finitely presented contravariant functors $\aa \to \mod k$.  Denote by $\aa'$ the full additive subcategory of $\aa$ generated by the set $\LL' = \LL \setminus \{L_1\}$.

\begin{proposition}
The category $\HH$ has only finitely many nonhomogeneous simple tubes.
\end{proposition}

\begin{proof}
Consider the full additive subcategory $\aa' \subset \HH$ generated by the objects $L$ and $L^{(x)}_1$, for each nonhomogeneous simple tube $x$.  Following \cite[Theorem 1.6]{AuslanderReiten75} (see also \cite[Proposition 4.2]{vanRoosmalen06}) we know that $\aa'$ is semi-hereditary (i.e. $\mod \aa'$ is abelian and hereditary) if and only $\End A$ is a hereditary algebra for all objects $A \in \Ob \aa'$.  This last statement is true since each such algebra is either semi-simple (if $L$ is not a direct summand of $A$) or Morita equivalent to the path algebra of a quiver
$$\xymatrix@R=0pt{& \bullet \\
\bullet \ar[ur] \ar[dr]& \vdots \\
& \bullet}$$
This last algebra is of wild type if $Q$ has at least 6 vertices, and thus there is an indecomposable object $Z \in \mod A$ with $\chi(Z,Z) = \dim \Hom(Z,Z) - \dim \Ext(Z,Z) < 0$ and the existence of a fully faithful functor $\Db \mod (\End A) \to \Db \HH$ shows that $\Db \HH$ also has such an object.  However, every object of $\HH$ lies in a generalized standard tube, and thus $\chi(X,X) \geq 0$ for all indecomposables $X$ in $\Db \HH$.

We see that $Q$ cannot have more than 5 vertices, and hence $\HH$ cannot have more than 4 nonhomogeneous simple tubes.
\end{proof}

Denote by $S_L$ the simple representation $\LL(-,L) / \rad(-,L)$ associated to $L$ and by $S_{L_1}$ the simple representation $\LL(-,L_1) / \rad(-,L_1)$ associated to $L_1$.  If $\LL$ has infinitely many arms, then $\dim \Ext^2(S_{L_1}, S_L) = \infty$, contradicting the Hom-finiteness of $\Db \AA$.

We will gather these results in following proposition.

\begin{proposition}\label{proposition:TiltingObject}
Let $\AA$ be an indecomposable abelian hereditary category which is fractionally Calabi-Yau of dimension 1 but not $1$-Calabi-Yau.  If the Auslander-Reiten quiver of $\AA$ consists of more than one tube, then $\Db \AA$ admits a tilting object.
\end{proposition}

\subsection{Proof of classification}

\begin{theorem}\label{theorem:FractionallyCY}
Let $\AA$ be an indecomposable abelian hereditary category which is fractionally Calabi-Yau of dimension $d$, but not $1$-Calabi-Yau then $\AA$ is derived equivalent to either
\begin{enumerate}
  \item the category of finite dimensional modules $\mod Q$ over a Dynkin quiver $Q$, or
  \item the category of nilpotent representations $\nilp \tilde{A}_n$ where $\tilde{A}_n$ has cyclic orientation and $n \geq 1$, or
  \item the category of coherent sheaves $\coh \bX$ over a weighted projective line of tubular type.
\end{enumerate}
\end{theorem}

\begin{proof}
By Proposition \ref{proposition:FractionalCYDim}, we know that $d \leq 1$.  If $d \not= 1$, then $\t^n \cong [m-n]$ where $n > 0$ and $m-n < 0$.  Thus for every indecomposable $X \in \Ob \AA$, there is a $k < n$ such that $\t^k X$ is a projective object.  As such, $X$ is a preprojective object and in particular it is directing.  Such categories in which every object is directing have been classified in \cite{vanRoosmalen06}.  We see that $\AA$ is equivalent to the category of finite dimensional representations $\mod Q$ over a Dynkin quiver $Q$.

We are left with the case where $d=1$.  In this case, every Auslander-Reiten component of $\AA$ is a standard tube (see Proposition \ref{proposition:LS}).  If there is only one such tube, $\AA$ is equivalent to category of nilpotent representations $\nilp \tilde{A}_n$ where $\tilde{A}_n$ has cyclic orientation.  This category is 1-Calabi-Yau if and only if $n=0$.

In case there is more than one tube,  Proposition \ref{proposition:TiltingObject} shows that $\AA$ admits a tilting object.

Invoking Theorem \ref{theorem:TiltingObject} shows that $\AA$ is either derived equivalent to $\mod A$ for a finite dimensional algebra $A$, or to $\coh \bX$ for a weighted projective line $\bX$.  Since every Auslander-Reiten component of $\AA$ is a standard tube, we are in the latter case where $\bX$ is of tubular type.
\end{proof}

\begin{remark}
In the same spirit as Corollary \ref{corollary:CalabiYauOnObjects}, we note that we do not use $\bS^n \cong [m]$, but only $\bS^n X \cong X[m]$, for all $X \in \Ob \Db \AA$, thus that $\AA$ satisfies a fractionally Calabi-Yau property \emph{on objects} alone.  From Theorem \ref{theorem:FractionallyCY} we infer that $\AA$ is fractionally Calabi-Yau.
\end{remark}

\begin{remark}
By combining Theorem \ref{theorem:FractionallyCY} with the classification of hereditary Calabi-Yau categories in \cite{vanRoosmalen08} we obtain Theorem \ref{theorem:Introduction}.
\end{remark}
\section{Derived equivalences}

Although the classification in Theorem \ref{theorem:FractionallyCY} is up to derived equivalence, we may as in \cite{vanRoosmalen08} give a classification of all hereditary fractionally Calabi-Yau categories up to equivalence.

If $\AA$ is a hereditary fractionally Calabi-Yau category of fractional dimension less than 1, then $\AA \cong \rep Q$ for a Dynkin quiver $Q$.  Every hereditary category derived equivalent to $\AA$ is of the same form (\cite{vanRoosmalen06}).

When $\AA$ is a hereditary fractionally Calabi-Yau category of fractional dimension 1, then it suffices to give all aisles $\UU$ in $\Db \AA$ which are closed under successors.  By the fractionally Calabi-Yau property, such aisles are $\t$-invariant.

Since we are interested in bounded $t$-structures, we may assume there in an $n \in \bZ$ such that $\UU$ contains a nonzero object of $\AA[n] \subseteq \Db \AA$ but no nonzero objects of $\AA[n-1]$.  Up to shifts, we may assume $n=0$. 

Such a $\tau$-invariant aisle can be described by a split torsion theory on $\AA$. A \emph{torsion theory} on $\AA$, $(\FF,\TT)$, is a pair of full additive subcategories of $\AA$, such that $\Hom(\TT,\FF)=0$ and having the additional property that for every $X \in \Ob \AA$ there is a short exact sequence
$$0 \longrightarrow T \longrightarrow X \longrightarrow F \longrightarrow 0$$
with $F \in \FF$ and $T \in \TT$.

We will say the torsion theory $(\FF,\TT)$ is \emph{split} if $\Ext(\FF,\TT)=0$.  In case of a split torsion theory we obtain, by \emph{tilting}, a hereditary category $\HH$ derived equivalent to $\AA$ with an induced split torsion theory $(\TT,\FF[1])$.

When $\AA \cong \nilp \tilde{A}_{n}$, the only split torsion theories are those with either $\TT = 0$ or $\FF = 0$.

\begin{proposition}\label{proposition:SplitStructures}
Let $\bX$ be a weighted projective line of tubular type or an elliptic curve.  All split torsion theories may be described as follows.  Let $\theta \in \bR \cup \{\infty\}$.  Denote by $\AA_{> \theta}$ and $\AA_{\geq \theta}$ the full additive subcategory of $\AA$ generated by all indecomposables $\EE$ with $\mu(\EE) > \theta$ and $\mu(\EE) \geq \theta$, respectively.  All full exact extension-closed subcategories $\TT$ of $\AA$ with $\AA_{\geq \theta} \subseteq \TT \subseteq \AA_{> \theta} \subseteq \AA$ give rise to a torsion theory $(\FF,\TT)$, with $\FF$ the full subcategory given by all objects $F \in \AA$ such that $\Hom_\AA(\TT,F) = 0$.
\end{proposition}

\begin{proof}
It follows from the description of $\coh \bX$ given in \S\ref{section:Examples} that if an indecomposable object $X \in \Ob \coh \bX$ lies in $\FF$, that all indecomposable $Y$ with $\mu Y > \mu X$ also lies in $\FF$.  The case of an elliptic curve can be found in \cite{Burban06, vanRoosmalen08}.
\end{proof}

\begin{remark}
The category $\TT$ from Proposition \ref{proposition:SplitStructures} is thus completely determined by a subset of tubes with objects of slope $\theta$.
\end{remark}

\begin{example}
We give some examples of torsion theories.  In here $\HH$ always stands for the category tilted with respect to the described torsion theory.  We will assume $\AA = \coh \bX$ where $\bX$ is a weighted projective line of tubular weight; similar torsion theories for the category of coherent sheaves on an elliptic curve have been considered in \cite[Example 4.9]{vanRoosmalen08}.
\begin{enumerate}
\item If $\theta \in \bQ \cup \{\infty\}$ and $\TT = \AA_{> \theta}$, then it follows from \cite[Proposition 10.8]{Lenzing07} that the tilted category $\HH$ is equivalent to $\coh \bX$.
\item If $\theta \in \bR \setminus \bQ$ and $\TT = \AA_{> \theta} = \AA_{\geq \theta}$ then $\HH$ does not have noetherian objects.  Indeed, one can use either a proof similar to \cite[Proposition 3.1]{Polishchuk04b} or use Proposition \ref{proposition:SimplesInTubes} to see that $\HH$ does not have any simple objects (since every noetherian object must map nonzero to a simple object, this is a contradiction).
\item If $\theta \in \bQ \cup \{\infty\}$ and $\TT = \AA_{\geq \theta}$, then $\HH$ is dual to $\AA$.  One can see this by considering a tilting object in $\HH$.
\end{enumerate}
\end{example}

In the following theorem, we will use an explicit description of the category $\coh \bX$ where $\bX$ is a weighted projective line of tubular type or an elliptic curve.  For a weighted projective line, the description can be found in \S\ref{subsection:WeightedProjectivesLines}.  For an elliptic curve, we refer to \cite{Atiyah57} or for a more recent account to \cite{BruningBurban07}.

\begin{theorem}\label{theorem:Equivalence}
Let $\AA$ be a connected hereditary fractionally Calabi-Yau category.  Then $\AA$ is equivalent to either
\begin{itemize}
\item the category of finite dimensional representations of a Dynkin quiver, or
\item the category of finite dimensional nilpotent representations of an $\tilde{A}_n$-quiver with cyclic orientation, or
\item a category obtained from tilting $\coh \bX$ with a split torsion theory described in Proposition \ref{proposition:SplitStructures} where $\bX$ is a weighted projective line of tubular type or an elliptic curve.
\end{itemize}
\end{theorem}

\begin{proof}
We only need to consider the last case, namely where $\AA$ is derived equivalent to $\coh \bX$ where $\bX$ is a weighted projective line of tubular type or an elliptic curve, thus $\AA$ is the heart of a split $t$-structure $(\DD^{\leq 0},\DD^{\geq 0})$ on $\Db \coh \bX$.

Let $X \in \ind \Db \coh \bX$.  Since the $t$-structure is bounded, there is an $n \gg 0$ so that $X[n] \in \DD^{\leq 0}$ and $X[-n] \in \DD^{\geq 0}$.  In particular, there is an $n \in \bZ$ so that $X[n] \in \DD^{\leq 0}$ but $X[n-1] \not\in \DD^{\leq 0}$.

It follows from the description of morphisms in $\coh \bX$ that $\HH[n+1] \subset \DD^{\leq 0}$.  Let $m \in \bZ$ be the smallest integer such that $\HH[m] \subset \DD^{\leq 0}$, so there is an $Y \in \ind \HH[m-1]$ with $Y \not\in \DD^{\leq 0}$, and because the $t$-structure is split, we have $Y \in \DD^{\geq 0}$.  Again using the description of the morphisms we have in $\coh \bX$, we know that $\HH[m-2] \subset \DD^{\geq 0}$.  The split $t$-structure is thus uniquely determined by $\TT = \HH[m-1] \cap \DD^{\leq 0}$, which defines a split torsion theory on $\HH$.
\end{proof}

\def\cprime{$'$}
\providecommand{\bysame}{\leavevmode\hbox to3em{\hrulefill}\thinspace}
\providecommand{\MR}{\relax\ifhmode\unskip\space\fi MR }
\providecommand{\MRhref}[2]{%
  \href{http://www.ams.org/mathscinet-getitem?mr=#1}{#2}
}
\providecommand{\href}[2]{#2}


\begin{thebibliography}{10}

\bibitem{ArtinZhang94}
M.~Artin and Zhang J.J., \emph{Noncommutative projective schemes}, Adv. Math.
  \textbf{109} (1994), no.~2, 228--287.

\bibitem{Atiyah57}
M.~F. Atiyah, \emph{Vector bundles over an elliptic curve}, Proc. London Math.
  Soc. (3) \textbf{7} (1957), 414--452.

\bibitem{AuslanderReiten75}
Maurice Auslander and Idun Reiten, \emph{Stable equivalence of dualizing
  {$R$}-varieties. {II}. {H}ereditary dualizing {$R$}-varieties}, Advances in
  Math. \textbf{17} (1975), 93--121.

\bibitem{ARS}
Maurice Auslander, Idun Reiten, and Sverre~O. Smal{\o}, \emph{Representation
  theory of {A}rtin algebras}, Cambridge Studies in Advanced Mathematics,
  vol.~36, Cambridge University Press, Cambridge, 1995.

\bibitem{Beilinson78}
A.~A. Be{\u\i}linson, \emph{Coherent sheaves on {${\bf P}^{n}$} and problems in
  linear algebra}, Funktsional. Anal. i Prilozhen. \textbf{12} (1978), no.~3,
  68--69.

\bibitem{BergVanRoosmalen10}
Carl~Fredrik Berg and Adam-Christiaan van Roosmalen, \emph{Hereditary
  categories with {S}erre duality which are generated by preprojectives},
  preprint.

\bibitem{Bondal89}
A.~I. Bondal, \emph{Representations of associative algebras and coherent
  sheaves}, Izv. Akad. Nauk SSSR Ser. Mat. \textbf{53} (1989), no.~1, 25--44.

\bibitem{BondalKapranov89}
A.~I. Bondal and M.~M. Kapranov, \emph{Representable functors, {S}erre
  functors, and reconstructions}, Izv. Akad. Nauk SSSR Ser. Mat. \textbf{53}
  (1989), no.~6, 1183--1205, 1337.

\bibitem{BruningBurban07}
Kristian Br{\"u}ning and Igor Burban, \emph{Coherent sheaves on an elliptic
  curve}, Interactions between homotopy theory and algebra, Contemp. Math.,
  vol. 436, Amer. Math. Soc., Providence, RI, 2007, pp.~297--315.

\bibitem{Burban06}
Igor Burban and Bernd Kreussler, \emph{Derived categories of irreducible
  projective curves of arithmetic genus one}, Compositio Math. \textbf{142}
  (2006), 1231--1262.

\bibitem{ChenHenning09}
Xiao-Wu Chen and Henning Krause, \emph{Introduction to coherent sheaves on
  weighted projective lines}, preprint.

\bibitem{CibilsZhang09}
Claude Cibils and Pu~Zhang, \emph{Calabi-{Yau} objects in triangulated
  categories}, Trans. Amer. Math. Soc. \textbf{361} (2009), no.~12, 6501--6519.

\bibitem{Gabriel62}
P.~Gabriel, \emph{Des categories abeliennes}, Bull. Soc. Math. France
  \textbf{90} (1962), 323--448.

\bibitem{Gabriel73}
\bysame, \emph{Indecomposable representations {II}}, Symposia Math. Inst. Naz.
  Alta Math. \textbf{11} (1973), 81--104.

\bibitem{GeigleLenzing87}
Werner Geigle and Helmut Lenzing, \emph{A class of weighted projective lines
  arising in representation theory of finite-dimensional algebras},
  Singularities, representation of algebras, and vector bundles (Lambrecht,
  1985), Lecture Notes in Math., vol. 1273, Springer, Berlin, 1987,
  pp.~265--297.

\bibitem{GorodentsevRudakov87}
A.~L. Gorodentsev and A.~N. Rudakov, \emph{Exceptional vector bundles on
  projective spaces}, Duke Math. J. \textbf{54} (1987), 115--130.

\bibitem{Happel01}
Dieter Happel, \emph{A characterization of hereditary categories with tilting
  object}, Invent. Math. \textbf{144} (2001), no.~2, 381--398.

\bibitem{KashiwaraShapira06}
Masaki Kashiwara and Pierre Schapira, \emph{Categories and sheaves},
  Grundlehren der Mathematischen Wissenschaften [Fundamental Principles of
  Mathematical Sciences], vol. 332, Springer-Verlag, Berlin, 2006.

\bibitem{KellerVossieck88}
B.~Keller and D.~Vossieck, \emph{Aisles in derived categories}, Bull. Soc.
  Math. Belg. S\'er. A \textbf{40} (1988), no.~2, 139--253, Deuxi{\`e}me
  Contact Franco-Belge en Alg{\`e}bre (Faulx-les-tombes, 1987).

\bibitem{Keller05}
Bernhard Keller, \emph{On triangulated orbit categories}, Doc. Math.
  \textbf{10} (2005), 551--581 (electronic).

\bibitem{Kussin09}
Dirk Kussin, \emph{Noncommutative curves of genus zero: related to finite
  dimensional algebras}, Mem. Amer. Math. Soc. (2009), no.~942, x+128.

\bibitem{Lenzing97}
Helmut Lenzing, \emph{Hereditary {N}oetherian categories with a tilting
  complex}, Proc. Amer. Math. Soc. \textbf{125} (1997), 1893--1901.

\bibitem{Lenzing07}
\bysame, \emph{Hereditary categories}, Handbook of Tilting Theory
  (Lidia~Angeleri H{\"u}gel, Dieter Happel, and Henning Krause, eds.), London
  Mathematical Society Lecture Notes Series, vol. 332, Cambridge University
  Press, Cambridge, 2007, pp.~105--146.

\bibitem{LoVdB06}
Wendy Lowen and Michel Van~den Bergh, \emph{Deformation theory of abelian
  categories}, Trans. Amer. Math. Soc. \textbf{358} (2006), no.~12, 5441--5483
  (electronic).

\bibitem{Meltzer97}
Hagen Meltzer, \emph{Tubular mutations}, Colloq. Math. \textbf{74} (1997),
  no.~2, 267--274.

\bibitem{Polishchuk04b}
A.~Polishchuk, \emph{Noncommutative two-tori with real multiplication as
  noncommutative projective varieties}, J. Geom. Phys. \textbf{50} (2004),
  no.~1-4, 162--187.

\bibitem{Polishchuk05}
\bysame, \emph{Noncommutative proj and coherent algebras}, Math. Res. Lett.
  \textbf{12} (2005), no.~1, 63--74.

\bibitem{ReVdB02}
I.~Reiten and M.~Van~den Bergh, \emph{Noetherian hereditary abelian categories
  satisfying {S}erre duality}, J. Amer. Math. Soc. \textbf{15} (2002), no.~2,
  295--366.

\bibitem{Ringel84}
Claus~Michael Ringel, \emph{Tame algebras and integral quadratic forms},
  Lecture Notes in Mathematics, vol. 1099, Springer-Verlag, Berlin, 1984.

\bibitem{Ringel05}
\bysame, \emph{Hereditary triangulated categories}, Compositio Math. (2005).

\bibitem{Seidel01}
Paul Seidel and Richard Thomas, \emph{Braid group actions on derived categories
  of coherent sheaves}, Duke Math. J. \textbf{108} (2001), no.~1, 37--108.

\bibitem{vanRoosmalen08}
Adam-Christiaan van Roosmalen, \emph{Abelian 1-{C}alabi-{Y}au categories}, Int.
  Math. Res. Not. IMRN \textbf{2008} (2008).

\bibitem{vanRoosmalen06}
\bysame, \emph{Classification of abelian hereditary directed categories
  satisfying {S}erre duality}, Trans. Amer. Math. Soc. \textbf{360} (2008),
  no.~5, 2467--2503.

\end{thebibliography}
\end{document}